\documentclass[11pt]{amsart}

\usepackage[utf8]{inputenc}
\usepackage{graphicx}
\usepackage{amsfonts}
\usepackage{amsmath}
\usepackage{amsthm}
\usepackage{amssymb}
\usepackage{pgfplots}
\pgfplotsset{compat=1.17}
\usepackage{graphicx}
\usepackage{float}
\usepackage{upgreek}
\usepackage{url}

\usepackage{siunitx}
\usepackage{booktabs}
\usepackage{textcomp, gensymb}
\usepackage{enumitem}
\usepackage{multicol}
\usepackage{mathrsfs}
\usepackage{algpseudocode}
\usepackage{adjustbox}
\usepackage{algorithm}
\usepackage{mathtools}
\usepackage{tikz-cd}
\usepackage[font=small,labelfont=bf]{caption}
\usepackage{wrapfig}
\usepackage[margin=1.24in]{geometry}
\usepackage{datetime}
\usepackage{comment}
\usepackage[
    backend=biber,
    style=ieee,
    citestyle=numeric-comp,
    sorting=nyt,
    dashed=false
]{biblatex}
\usepackage{hyperref}
\usepackage{xurl}
\hypersetup{
    colorlinks=true,
    linkcolor=blue,
    citecolor=red,
    breaklinks=true
}

\newdateformat{monthyeardate}{\monthname[\THEMONTH] \THEYEAR}

\title[Five ways to recover the symbol of a non-binary localization operator]{Five ways to recover the symbol of a\\ non-binary localization operator}

\author{Simon Halvdansson}
\address{Department of Mathematical Sciences, Norwegian University of Science and Technology, 7491 Trondheim, Norway.}
\email{simon.halvdansson@ntnu.no}

\date{\monthyeardate\today}

\theoremstyle{plain}
\newtheorem{theorem}{Theorem}[section]
\newtheorem*{theorem*}{Theorem}
\newtheorem{lemma}[theorem]{Lemma}
\newtheorem{proposition}[theorem]{Proposition}

\theoremstyle{definition}

\newtheorem{example}[theorem]{Example}

\theoremstyle{remark}
\newtheorem*{remark}{Remark}

\newcommand{\C}{\mathbb{C}}
\newcommand{\R}{\mathbb{R}}

\newcommand{\N}{\mathbb{N}}
\newcommand{\Var}{\operatorname{Var}}

\newcommand{\tr}{\operatorname{tr}}

\makeatletter
\newcommand{\vast}{\bBigg@{4}}
\newcommand{\Vast}{\bBigg@{5}}
\makeatother

\def\XXint#1#2#3{{\setbox0=\hbox{$#1{#2#3}{\int}$ }
		\vcenter{\hbox{$#2#3$ }}\kern-.6\wd0}}

\begin{document}
    \maketitle
    \begin{abstract}\vspace{-9mm}
        Five constructive methods for recovering the symbol of a time-frequency localization operator with non-binary symbol are presented, two based on earlier work and three novel methods. For the two derivative methods which have previously been applied to binary symbols, we propose a changed symbol estimator and provide additional estimates that show how we can recover non-binary symbols. The three novel methods each have their own advantages and are all applicable to non-binary symbols. Two of them rely on prescribing the input of the localization operator and examining the output, allowing for targeting of the part of the symbol one wishes to recover while the last one relies on spectral information about the operator. All five methods are also implemented numerically and evaluated with the code available.
	\vspace{6mm}
    \end{abstract}
	
    \renewcommand{\thefootnote}{\fnsymbol{footnote}}
    \renewcommand{\thefootnote}{\arabic{footnote}}

    \newcounter{SimonCounter}
    \newcommand{\simon}[1]{{\small \color{red}
    		\refstepcounter{SimonCounter}\textsf{[SH]$_{\arabic{SimonCounter}}$:{#1}}}}
    		
    \newcounter{FranzCounter}
    \newcommand{\franz}[1]{{\small \color{blue}
    		\refstepcounter{FranzCounter}\textsf{[FL]$_{\arabic{FranzCounter}}$:{#1}}}}

    \section{Introduction and main results}
    Arguably the main tool of time-frequency analysis is the \emph{short-time Fourier transform}, defined for a signal $\psi \in L^2(\R^d)$ and window $g \in L^2(\R^d)$ as
    \begin{align*}
        V_g \psi(x,\omega) = \int_{\R^d} \psi(t) \overline{g(t-x)}e^{-2\pi i \omega t}\,dt
    \end{align*}
    where the variables $x,\omega \in \R^d$ are referred to as the time and frequency, respectively. A standard result \cite{grochenig_book} states that this mapping can be inverted so that the signal $\psi$ can be recovered from $V_g\psi$ weakly as
    \begin{align*}
        \psi = \int_{\R^{2d}} V_g\psi(x,\omega) g(\cdot-x) e^{2\pi i \omega t}\,dx\,d\omega.
    \end{align*}
    Using a function $f : \R^{2d} \to \R$, usually referred to as the \emph{symbol} or \emph{mask}, we can weigh this reconstruction so that certain frequencies and time intervals have more or less priority than others. Formally, this happens via the \emph{localization operator}
    \begin{align}\label{eq:loc_op_def}
        A_f^g : \psi \mapsto \int_{\R^{2d}} f(x,\omega) V_g\psi(x,\omega) g(\cdot-x) e^{2\pi i \omega t}\,dx\,d\omega.
    \end{align}
    Such operators have applications in signal analysis \cite{Olivero2010, Kreme2021, Rajbamshi2019, Strohmer2006}, acoustics \cite{Hazrati2013, Wang2005, Chi2012}, pseudo-differential operators \cite{Hrmander1979, Grchenig2011}, physics \cite{Berezin1971, deGosson2011} and operator theory \cite{Luef2021, Grchenig2011, Fulsche2020} among others and their properties have been extensively studied \cite{CORDERO2007, Wong2002, Fernndez2006, Cordero2006}. Abreu and Dörfler \cite{Abreu2012} first considered the inverse problem of recovering the symbol $f$ from the localization operator $A_f^g$ through various measurements related to $A_f^g$ and this work has been continued in \cite{Abreu2012, Abreu2015, Romero2022, Luef2019_acc}. All these investigations have been focused on the case where $f$ is a \emph{binary mask}, i.e., $f : \R^{2d} \to \{ 0, 1 \}$. The main contribution of this paper is showing corresponding results for more general classes of $f$ as well as developing novel approaches to recovering $f$ which have not been considered before.
    
    There are several reasons to consider the case of non-binary symbols: If we view the inverse problem as a calibration, it is reasonable that imperfections in the system may cause the corresponding symbol to deviate from an intended binary design. Symbol discontinuities can also cause audible artifacts known as \emph{musical noise} in the audio setting \cite{Chi2012, Brons2012} and it is therefore beneficial to design those systems with non-binary symbols in the first place. In some audio filtering contexts where binary masks are currently used such as in \cite{Hazrati2013}, a non-binary value is associated to each time-frequency coordinate and a mask is then constructed by thresholding. This approach, while straight-forward, is unlikely to be optimal which has motivated the use of non-binary masks in similar systems \cite{Chi2012}. Audio processing systems may also include components such as preemphasis which cannot be represented using binary time-frequency masks. Moreover, localization operators can be identified with function-operator convolutions from quantum harmonic analysis \cite{Luef2018} and Gabor-Toeplitz operators \cite{Luef2021} and inverting the symbol to operator mapping is of independent theoretical interest in these settings. As localization operators with a fixed window function are dense in the class of trace-class operators \cite{Bayer2015, Luef2018}, a general operator can be completely described by its symbol which is unlikely to be binary-valued, hence non-binary symbol recovery provides a suitable setting for a general mapping from operators to functions.
    
    Below we state all of our main results before having established all the relevant notation. In particular, we formulate some results with function-operator convolutions $f \star S$, Wigner distributions $W(\varphi)$, the Feichtinger algebra $M^1(\R^d)$ and Cohen's class distributions $Q_S(\psi)$. These are all detailed and properly defined in Section \ref{sec:prelims} but hopefully the general idea of the theorems should be clear. If not, the reader can return to the formulations after finishing the preliminaries section.

    \subsection*{White noise}
    Our first result shows how a smooth, non-negative, real-valued symbol can be approximated by looking at the spectrograms of the images of white noise under the localization operator. Due to specifics of the approximation procedure, we can only estimate the square of the symbol, $f^2$, meaning that sign information is lost. In particular, our estimator for $f^2$, the so called \emph{average observed spectrogram}, is given by
    \begin{align}\label{eq:rho_def}
        \rho(z) = \frac{1}{K}\sum_{k=1}^K |V_\varphi(A_f^g \mathcal{N}_k)(z)|^2
    \end{align}
    where $(\mathcal{N}_k)_{k=1}^K$ are $K$ realizations of (complex) white noise and $\varphi$ is our \emph{reconstruction window} which does not necessarily have to coincide with $g$. This construction is from \cite{Romero2022} in the binary case and follows previous work on white noise approaches in time-frequency analysis which have recently received attention \cite{Abreu2022_gaf, Bardenet2021}. The notion and interpretation of white noise in this setting will be made precise in Section \ref{sec:white_noise}. We will show how, as $K \to \infty$, this estimator converges with high probability to
    \begin{align}\label{eq:vartheta}
        \vartheta(z) = \sum_m \lambda_m^2 |V_\varphi h_m(z)|^2
    \end{align}
    where $(\lambda_m)_m$ and $(h_m)_m$ are the eigenvalues and eigenfunctions of $A_f^g$, respectively. Using the framework of quantum harmonic analysis and asymptotics of products of localization operators, we will show how $\vartheta$ in turn is a good approximation of $f^2$.
    \begin{theorem}\label{theorem:L1_theorem}
        Let $f \in C^{d+2}_c(\R^{2d})$ be real-valued, $\rho$ be given by \eqref{eq:rho_def} with white noise variance $\sigma^2$, $g, \varphi \in \mathcal{S}(\R^d)$ with $\Vert g \Vert_{L^2} = \Vert \varphi \Vert_{L^2} = 1$ and define
        \begin{align*}
            B_1 &= \left[ \Vert G \Vert_{L^2}  + \left(\sum_{j=1}^{2d} \big\Vert \partial_j^{d+1} G \big\Vert_{L^2}^2 \right)^{1/2} \right],\\
            B_2 &= \left(\int_{\R^{2d}} \big|(\nabla f^2)(z)\big|\,dz\right) \left(\int_{\R^{2d}} |z| |V_\varphi g(z)|^2\,dz\right)
        \end{align*}
        where
        $$
        G(y,z) = f(y) \left( \sum_{|\alpha| = 1} \int_0^1 \partial^\alpha f(y+t(z-y))\,dt (z-y) \right) V_g g(y-z).
        $$
        Then there exists a $C > 0$ such that
        \begin{align}\label{eq:L1_PAC}
            \mathbb{P}\left( \int_{\R^{2d}} \left| \frac{\rho(z)}{\sigma^2} - f(z)^2 \right| \,dz > B + t \right) \leq \frac{\Vert f \Vert_{L^2}^2}{t\sqrt{K}} \left[ \frac{3\sqrt{\pi} }{2\sqrt{C}}\operatorname{erf}\big(\sqrt{CK}\big) + \frac{3}{C\sqrt{K}} e^{-CK } \right]
        \end{align}
        where $B = AB_1 + B_2$ and $A$ is a constant independent of $f$ and $g$.
    \end{theorem}
    The above theorem should be read as \textit{``the $L^1$ estimation error is bounded by $B$ with high probability provided $K$ is large enough''}. Note that the assumptions on $f$ in the theorem are not fulfilled by binary symbols. An intermediate step in the proof of the theorem (Proposition \ref{prop:white_noise_intermediate}) shows how the average observed spectrogram still approximates $f^2$ under looser conditions on $f$ but it does not supply a unified error estimate in the way that \eqref{eq:L1_PAC} does.

    \subsection*{Weighted accumulated Cohen's class}
    Next we discuss two approaches which rely on spectral data about the localization operator. These are also stated for \emph{mixed-state} localization operators, introduced and defined in Section \ref{sec:mixed_state_loc_op} below, as this stronger result follows directly from our methods. For the motivation of these operators beyond generalization, see \cite{Luef2019, Luef2019_acc}. We also state the corresponding statement for the ``pure'' localization operators discussed above.
    \begin{theorem}\label{theorem:accumulated}
        Let $f \in L^1(\R^{2d})$ be real-valued and with bounded variation and $S, T$ be positive trace-class operators with $\tr(S) = \tr(T) = 1$. Then if $f \star S = \sum_m \lambda_m (h_m \otimes h_m)$,
        \begin{align}\label{eq:accumulated_cohen}
            \sum_{m=1}^\infty \lambda_m Q_T(h_m)(z) = f * (S \star \check{T})(z)    
        \end{align}
        and the $L^1$ error can be bounded as
        $$
        \left\Vert \sum_{m=1}^N \lambda_m Q_T(h_m) - f \right\Vert_{L^1} \leq \operatorname{Var}(f) \int_{\R^{2d}} |z|(S \star \check{T})(z)\,dz + \sum_{m=N+1}^\infty |\lambda_m|.
        $$
        In particular, if $\varphi \in L^2(\R^d)$ with $\Vert \varphi \Vert_{L^2} = 1$ and $A_f^g = \sum_m \lambda_m (h_m \otimes h_m)$, then
        $$
        \left\Vert \sum_{m=1}^N \lambda_m |V_\varphi(h_m)|^2 - f \right\Vert_{L^1} \leq \operatorname{Var}(f) \int_{\R^{2d}} |z| |V_\varphi g(z)|^2\,dz + \sum_{m=N+1}^\infty |\lambda_m|.
        $$
        
    \end{theorem}
    Again, the notation $Q_S$ for Cohen's class is explained later in Section \ref{sec:tf_cohen} but reduces down to the spectrogram when $S$ is a rank-one operator as highlighted above.

    Similar sums have previously been investigated in \cite{Abreu2015} and \cite{Luef2019_acc} as accumulated spectrograms and Cohen's class distributions without the factor $\lambda_m$ in front of the Cohen's class distribution or spectrogram. In Section \ref{sec:weighted_accumulated_cohen}, we make the argument as to why the above formulation, which we refer to as the \emph{weighted} accumulated Cohen's class or spectrogram, is preferable and leads to smaller errors.

    \subsection*{Weighted accumulated Wigner distribution}   
    For the weighted accumulated Cohen's class as well as the white noise approach, we had to use a reconstruction window $\varphi$. In the next theorem, we are able to sidestep this. Essentially the weighted accumulated Wigner distribution is the Weyl symbol of the localization operator and its properties are easy to deduce from this perspective.
    \begin{theorem}\label{theorem:fourier_deconvolution}
        Let $f \in L^1(\R^{2d})$ be real-valued with bounded variation and $S = \sum_n s_n (g_n \otimes g_n)$ a positive trace-class operator. Then if $f \star S = \sum_m \lambda_m (h_m \otimes h_m)$,
        \begin{align*}
            \sum_m \lambda_m W(h_m)(z) &= f * \sum_n s_n W(g_n)(z)
        \end{align*}
        where $W(g)$ is the Wigner distribution of $g$. In particular, if $S = g \otimes g$ so that $f \star S = A_f^g$,
        \begin{align*}
            \sum_m \lambda_m W(h_m)(z) &= f * W(g)(z).
        \end{align*}
        Moreover, if the window functions $(g_n)_n$ are in the Schwartz space, the convergence is in $L^1$ with the error bounds
        \begin{align*}
            \left\Vert \sum_m \lambda_m W(h_m) - f \right\Vert_{L^1} &\leq \Var(f) \int_{\R^{2d}} |z| \left| \sum_n s_n W(g_n)(z) \right|\,dz
        \end{align*}
        and the corresponding statement holds for the rank-one case $S = g \otimes g$.
    \end{theorem}
    \subsection*{Plane tiling}
    Next, we discuss an approach based on noting that adding up all the spectrograms of an orthonormal basis yields the function which is identically one in a manner that be likened to a tiling of phase space via a partition of unity. If we then apply our localization operator with symbol $f$ to each basis element, this tiling should only make a contribution proportional to the size of $f^2$. This intuition turns out to be correct and is quantified in the following theorem. 
    \begin{theorem}\label{theorem:plane_tiling}
        Let $f \in C_c^{d+2}(\R^{2d})$ be real-valued and $g,\varphi \in \mathcal{S}(\R^{2d})$ with $\Vert g \Vert_{L^2} = \Vert \varphi \Vert_{L^2} = 1$. Then for any orthonormal basis $\{ e_n \}_n$ of $L^2(\R^d)$,
        \begin{align*}
            \left\Vert \sum_n |V_\varphi(A_f^g e_n)|^2 - f^2\right\Vert_{L^1} \leq B
        \end{align*}
        where $B$ is as in Theorem \ref{theorem:L1_theorem}.
    \end{theorem}
    The reason the above error estimate is so similar to that in Theorem \ref{theorem:L1_theorem} is that it turns out that the plane tiling estimator is precisely the limit $\vartheta$ \eqref{eq:vartheta} which the average observed spectrogram converges to pointwise as $K\to \infty$ almost surely.
    
    \subsection*{Gabor space projection}
    The last method is fundamentally different from those above in that it estimates $f$ pointwise in a parallelizable way. Given a point $z \in \R^{2d}$, we estimate $f(z)$ as follows: Translate the window function $g$ by $z$ so that $V_g(\pi(z)g)$ is centered at $z$ and takes the value $1$ there, i.e., $V_g(\pi(z)g)(z) = 1$. Next apply the localization operator to $\pi(z) g$ so that the value of the STFT at $z$ is scaled by approximately $f(z)$.
    The reconstruction window used can differ from the original window $g$ as long as $V_\varphi g$ reaches its maximum close to $0$. As we show in Section \ref{sec:gabor_projection}, when $\varphi = g$ this procedure may be interpreted as projecting $f \cdot V_gg(\cdot-z)$ onto the Gabor space $V_g(L^2)$ which is the reason for the name.

    \begin{theorem}\label{theorem:gabor_projection}
        Let $f \in L^1(\R^{2d})$ and $g, \varphi \in L^2(\R^d)$ with $\Vert g \Vert_{L^2} = \Vert \varphi \Vert_{L^2} = 1$. Then
        \begin{align*}
            V_\varphi(A_f^g(\pi(z)\varphi))(z) = f * |V_\varphi g|^2(z)
        \end{align*}
        and consequently the estimator satisfies the error estimate
        \begin{align*}
            \big\Vert V_\varphi(A_f^g(\pi(\cdot)\varphi)) - f \big\Vert_{L^1} \leq \operatorname{Var}(f)\int_{\R^{2d}} |z| |V_\varphi g(z)|^2\,dz.
        \end{align*}
    \end{theorem}
    Due to the pointwise nature of this method, it can be used in conjunction with the white noise or plane tiling estimator which only estimates $f^2$ to supply sign information for $f$.

    Note that the above estimator is identical to that for the weighted accumulated spectrogram \eqref{eq:accumulated_cohen}. Consequently they can be combined as the weighted accumulated spectrogram is far more susceptible to numerical instabilities as discussed later in Section \ref{sec:num_acc_spec}. Moreover, the estimator is the convolution of the spectrogram and then kernel $|V_\varphi g|^2$ and so if we know $g$ and can set $\varphi = g$, we can recover $f$ precisely by a deconvolution procedure. This is of course a very unstable procedure but we implement it numerically in Section \ref{sec:gabor_projection_details} to show that it is feasible.

    \subsubsection*{Outline}
    In Section \ref{sec:prelims}, we go through all of the required preliminaries to follow the proofs of the above theorems. Sections \ref{sec:white_noise_recovery}, \ref{sec:spectral_recovery}, \ref{sec:plane_tiling} and \ref{sec:gabor_projection} are devoted to discussing the details, proofs and appropriate considerations for all of the recovery methods outlined above. Lastly in Section \ref{sec:numerical_implementation}, all of the five methods are implemented in MATLAB using the Large Time/Frequency Analysis Toolbox \cite{ltfatnote030} with the code available on GitHub\footnote{\url{https://github.com/SimonHalvdansson/Localization-Operator-Symbol-Recovery}}. We also discuss implementation details and compare the performance of the different methods.

    \subsubsection*{Notational conventions}
    The Schatten $p$-class of operators with singular values in $\ell^p$ will be denoted by $\mathcal{S}^p$ while the larger class of bounded operators on $L^2(\R^d)$ will be denoted by $\mathcal{L}(L^2)$. The adjoint of such an operator $A$ will be denoted by $A^*$ and we will write $\check{A} = PAP$ where $P$ is the parity operator $P : f(t) \mapsto f(-t)$. For the open ball centered at $z$ with radius $r$ we will write $B_r(z)$ and for a function $f$ of several variables, we will write $\partial_j^n f$ for the $n$:th derivative in the $j$:th variable. We will also use a multiindex $\alpha = (\alpha_1, \dots, \alpha_d)$ for the derivative $\partial^\alpha f = \partial_{\alpha_1} \cdots \partial_{\alpha_d} f$ and denote by $C^n$ the set of functions $f$ for which $\partial^\alpha f$ is continuous for all $\alpha \in \N^d$ with $|\alpha| \leq n$. The associated subspace $C^n_c$ will specify those functions which have compact support. For the Schwartz functions on $\R^d$ we will write $\mathcal{S}(\R^d)$ while indicator functions of sets $\Omega$ will be denoted by $\chi_\Omega$. Inner products with no subscript will always refer to $L^2(\R^d)$ inner products so that $\langle \cdot, \cdot \rangle = \langle \cdot , \cdot \rangle_{L^2(\R^d)}$. For matrices of size $N \times M$ with entries in $\C$, we will write $\C^{N\times M}$.
    
    \section{Preliminaries}\label{sec:prelims}

    \subsection{Time-frequency analysis}
    We highlight some important facts from time-frequency analysis which we will have use for, in a compact form. For a more complete introduction the reader is referred to \cite{grochenig_book, Wong2002}.
    
    \subsubsection{Short-time Fourier transform}
    One of the main ideas underlying time-frequency analysis is that real world signals are correlated in time and frequency. This is utilized by the short-time Fourier transform (STFT) which allows us to analyze the frequency content of a signal $\psi$ around a specific time by multiplying $\psi$ by a window function $g$ which is well-localized in time. It is defined as
    \begin{align*}
        V_g \psi(z) = \langle \psi, \pi(z)g \rangle = \int_{\R^d}\psi(t) \overline{g(t-x)}e^{-2\pi i \omega t}\,dt.
    \end{align*}
    where the \emph{time-frequency shift} $\pi(z)$ is a projective unitary square integrable representation acting as $\pi(z) \psi(t) = \pi(x,\omega) \psi(t) = e^{2\pi i \omega t} \psi(t-x)$ and the point $z = (x, \omega) \in \R^{2d}$ is said to belong to \emph{phase space}. A standard result known as \emph{Moyal's formula} states that we can compute inner products of two signals using their associated short-time Fourier transforms as
    $$
    \langle V_{g_1}\psi_1, V_{g_2}\psi_2\rangle_{L^2(\R^{2d})} = \langle \psi_1, \psi_2 \rangle \overline{\langle g_1, g_2 \rangle}. 
    $$
    In particular, the mapping $V_g : L^2(\R^d) \to L^2(\R^{2d})$ is an isometry provided $\Vert g \Vert_{L^2} = 1$. As a consequence of the above relation, we have the reconstruction formula from the introduction which tells us how we can get a signal back from its STFT
    \begin{align}\label{eq:stft_reconstruction}
	\psi = \int_{\R^{2d}} V_g \psi(z) \pi(z) g\,dz.
    \end{align}
    Often in application, the square modulus $|V_g\psi|^2$ of the STFT, the \emph{spectrogram}, is used as it is real-valued and non-negative and thus represents the \emph{energy} distribution of the signal.
 
    \subsubsection{Localization operators}
    Provided we want to localize the support of a signal $\psi$ to a region of phase space, an obvious idea is to limit the reconstruction in \eqref{eq:stft_reconstruction} to some subset $\Omega \subset \R^{2d}$. By the uncertainty principle, this is impossible but works up to a small error depending on the size of $\Omega$. Generalizing this idea, we can add a weighing factor, or \emph{symbol}, $f \in L^1(\R^{2d})$ to the reconstruction \eqref{eq:stft_reconstruction} which tells us how much we want to reconstruct different parts of the phase space representation of $\psi$. Formally, we write the application of the localization operator $A_f^g$ to $\psi$ as 
    \begin{align}
        A_f^g \psi (t) = \int_{\R^{2d}} f(z) V_g \psi(z) \pi(z) g(t)\,dz.
    \end{align}
    Localization operators in the time-frequency context were originally investigated by I. Daubechies \cite{daubechies1988_loc, Daubechies1990}.

    \subsubsection{Gabor spaces}
    The range of the short-time Fourier transform with a specific window $g$ is denoted by $V_g(L^2) \subset L^2(\R^{2d})$ and called the \emph{Gabor space} associated to $g$. These spaces are reproducing kernel Hilbert spaces which fulfill
    $$
    F = F * V_g g \qquad \text{ for all }F \in V_g(L^2).
    $$
    We already mentioned that $V_g^* V_g$ is the identity but reversed composition, $V_g V_g^* : L^2(\R^{2d}) \to V_g(L^2)$ is the orthogonal projection onto the Gabor space, often denoted by $P_{V_g(L^2)}$.

    \subsubsection{Modulation spaces}
    Feichtinger's algebra $M^1(\R^d)$, originally introduced in \cite{Feichtinger1981}, is defined as the set of tempered distributions $\psi$ such that $V_g \psi \in L^1(\R^{2d})$ where $g$ is a Schwartz window function. It is a special case of the \emph{modulation spaces} \cite{Feichtinger1981, FEICHTINGER1989307} which are defined by integrability properties of short-time Fourier transforms and have the convenient property that they are independent of the window function $g$ used.

    \subsubsection{Cohen's class of time-frequency distributions}\label{sec:tf_cohen}
    There are several quadratic time-frequency distributions with properties similar to those of the spectrogram. Those that fulfill some basic desirable properties are commonly referred to as \emph{Cohen's class distributions} \cite{Cohen1989} and include the spectrogram as a special case. They can all be written as
    $$
    Q_\Phi (\psi) =  \Phi * W(\psi) 
    $$
    where $\Phi$ is a tempered distribution and $W(\psi) = W(\psi, \psi)$ is the \emph{Wigner distribution} of $\psi$, defined as
    \begin{align}\label{eq:wigner_def}
        W(\psi, \phi)(x, \omega) = \int_{\R^d}\psi(t+x/2)\overline{\phi(t-x/2)}e^{-2\pi i \omega \cdot t}\,dt.
    \end{align}

    \subsection{Quantum harmonic analysis}
    The theory of quantum harmonic analysis, first developed by Werner in \cite{Werner1984}, will play a central role in several of our main proofs. Its main components are definitions of convolutions between functions and operators and pairs of operators. For a function $f$ and two operators $T$ and $S$, these take the form
    \begin{align}\label{eq:func_op_op_op_def}
        f \star S = \int_{\R^{2d}} f(z) \pi(z) S \pi(z)^*\,dz,\qquad T \star S(z) = \tr\big(T \pi(z) \check{S} \pi(z)^*\big)
    \end{align}
    where the first integral should be interpreted as a Bochner integral and $\check{S} = PSP$. Note in particular that with this formalism, $f \star S$ is an operator while $T \star S$ is a function. As we will see below, both of these definitions satisfy versions of Young's inequality which we will make use of. However, we first compute the prototypical rank-one function-operator and operator-operator convolutions since these serve as our main motivation for using the framework of quantum harmonic analysis. We remind the reader that a rank-one operator $\psi \otimes \phi$ acts as $f \mapsto \langle f, \phi\rangle \psi$.
    \begin{example}
        The function-operator convolution $f \star (g \otimes g)$ for $f \in L^1(\R^{2d})$ and $g \in L^2(\R^d)$ is precisely the localization operator $A_f^g$. Indeed,
        \begin{align*}
            f \star (g \otimes g) \psi &= \int_{\R^{2d}} f(z) \pi(z) (g \otimes g) \pi(z)^* \psi \,dz\\
            &= \int_{\R^{2d}} f(z) \langle \pi(z)^* \psi , g \rangle \pi(z) g \,dz\\
            &= \int_{\R^{2d}} f(z) V_g \psi(z) \pi(z) g \,dz = A_f^g \psi.
        \end{align*}
    \end{example}
    \begin{example}\label{example:qha_rank_one_conv}
        The simplest case of operator-operator convolutions reduces down to the spectrogram. For $\psi, \phi \in L^2(\R^d)$, we have that
        \begin{align*}
            (\psi \otimes \psi) \star (\phi \otimes \phi)\check{\,} \,(z) &= \tr\big( (\psi \otimes \psi) \pi(z) (\phi \otimes \phi)\pi(z)^* \big)\\
            &= \sum_n \big\langle (\psi \otimes \psi) \pi(z) (\phi \otimes \phi)\pi(z)^* e_n, e_n \big\rangle\\
            &= \sum_n \langle e_n, \pi(z)\phi \rangle  \langle \pi(z)\phi, \psi \rangle \langle \psi, e_n \rangle\\
            &= |\langle \psi, \pi(z) \phi \rangle|^2 = |V_\phi \psi(z)|^2
        \end{align*}
        where $(e_n)_n$ was an arbitrary orthonormal basis used to compute the trace.
    \end{example}
    Many properties of function-operator and operator-operator convolutions are analogues of corresponding statements for classical function-function convolutions as we will see below.
    
    Just as the integral is replaced by the trace in the definition of operator-operator convolutions \eqref{eq:func_op_op_op_def}, when measuring the size of operators, we will use the Schatten $p$-norms which are defined as
    $$
    \Vert A \Vert_{\mathcal{S}^p} = \tr(|A|^p)^{1/p}
    $$
    where $|A| = \sqrt{A^* A}$ is the absolute value of $A$. These norms induce the Schatten $p$-classes of operators, the most notable of which are the trace-class operators $\mathcal{S}^1$ and the Hilbert-Schmidt operators $\mathcal{S}^2$. As these operators are compact, they have a spectral decomposition of the form
    $$
    A = \sum_n a_n (\psi_n \otimes \phi_n)
    $$
    where $(\psi_n)_n$ and $(\phi_n)_n$ are orthonormal bases and $(a_n)_n$ is in $\ell^p$ if $A \in \mathcal{S}^p$. Note in particular that if $A$ is self-adjoint, we have that $\psi_n = \phi_n$ for all $n$.
    
    Next we collect some basic properties of these convolutions, the proofs of which can be found in \cite{Luef2018}.
    \begin{proposition}\label{prop:op_conv_properties}
    	Let $f,g \in L^1(\R^{2d}),\, S \in \mathcal{S}^p,\, T \in \mathcal{S}^q$ for $1 \leq p, q \leq \infty$ with $\frac{1}{p} + \frac{1}{q} = 1$ and $R \in \mathcal{S}^1$. Then
    	\begin{enumerate}[label=(\roman*)]
    		\item \label{item:func_op_adj}$(f \star S)^* = \bar{f} \star S^*$,
    		\item \label{item:ass1} $(f \star R) \star T = f * (R \star T)$,
    		\item \label{item:ass2} $(f * g) \star S = f  \star (g \star S)$,
    		\item $\Vert f \star S \Vert_{\mathcal{S}^p} \leq \Vert f \Vert_{L^1} \Vert S \Vert_{\mathcal{S}^p}$,
    		\item \label{item:func_op_bound}$\Vert h \star R \Vert_{\mathcal{S}^p} \leq \Vert h \Vert_{L^p} \Vert R \Vert_{\mathcal{S}^1}$,
    		\item \label{item:op_op_bound}$\Vert S \star R \Vert_{L^p} \leq \Vert S \Vert_{\mathcal{S}^p} \Vert R \Vert_{\mathcal{S}^1}$.
    	\end{enumerate}
    \end{proposition}

    In the next subsections, we dive deeper into some topics in quantum harmonic analysis which will be of use. For a more thorough introduction with more motivation and results, the reader is referred to \cite{Luef2018}.

    \subsubsection{Mixed-state localization operators}\label{sec:mixed_state_loc_op}
    Localization operators reconstruct a function with respect to a single window or pair of windows in the non self-adjoint case. This construction has been generalized to multiple windows by considering the function-operator convolution $f \star S$ which can be seen as a weighted sum of localization operators \cite{Luef2019}. Indeed, if $S = \sum_n s_n (g_n \otimes g_n) \in \mathcal{S}^1$, then
    $$
    f \star S = f \star \sum_n s_n (g_n \otimes g_n) = \sum_n s_n A_f^{g_n}.
    $$
    Later on in our results, we will need for our (mixed-state) localization operators to be self-adjoint. In view of Proposition \ref{prop:op_conv_properties} \ref{item:func_op_adj}, this requires the symbol $f$ to be real-valued and the window operator $S$ to be self-adjoint.

    \subsubsection{Fourier-Wigner transform}\label{sec:fourier_wigner_deconv}
    Another central tool of quantum harmonic analysis is the Fourier-Wigner transform, mapping operators to functions, defined for $S \in \mathcal{S}^1$ as
    \begin{align*}
        \mathcal{F}_W(S)(z) = e^{-\pi i x \omega} \tr(\pi(-z) S).
    \end{align*}
    Our interest in the Fourier-Wigner transform is primarily based on its convolution properties which mirror those of the classical Fourier transform. To state the relevant result, we first need to define the \emph{symplectic Fourier transform} which essentially is a rotated two-dimensional Fourier transform
    $$
    \mathcal{F}_\sigma(f)(z) = \int_{\R^{2d}} f(z') e^{-2\pi i \sigma(z,z')}\,dz'
    $$
    where $z = (x,\omega),\, z' = (x', \omega')$ and $\sigma(z,z') = \omega x' - \omega' x$ is the standard symplectic form. We can now state the result which is analogous to the classical convolution theorem.
    \begin{proposition}
    Let $f \in L^1(\R^{2d})$ and $S, T \in \mathcal{S}^1$. Then
    \begin{align*}
        \mathcal{F}_W(f \star S) &= \mathcal{F}_\sigma(f) \cdot \mathcal{F}_W(S),\\
        \mathcal{F}_\sigma(T \star S) &= \mathcal{F}_W(T) \cdot \mathcal{F}_W(S).
    \end{align*}
    \end{proposition}
    
    \subsubsection{Weyl quantization}\label{sec:weyl_quantization}
    A \emph{quantization procedure} provides a mapping between functions and operators such as the mapping $f \mapsto f \star S$ which we invert in this paper. In time-frequency analysis and quantum harmonic analysis, we often make use of \emph{Weyl quantization} which can be defined weakly as
    $$
    \big\langle L_f \psi, \phi \big\rangle = \big\langle f, W(\phi, \psi) \big\rangle
    $$
    where we refer to the mapping $f \mapsto L_f$ as the \emph{Weyl transform}. For the inverse mapping, meaning the function associated to the operator $S$, we write $a_S$ and call it the \emph{Weyl symbol} of $S$.
    
    Weyl quantization has a particularly nice formulation in quantum harmonic analysis where it can be written as
    $$
    a_S = \mathcal{F}_\sigma(\mathcal{F}_W(S)).
    $$
    In particular, it can be shown that $a_{\psi \otimes \phi} = W(\psi, \phi)$. It also holds that Weyl quantization is compatible with the convolutions of quantum harmonic analysis in the sense that
    \begin{align}\label{eq:weyl_quantization_conv_props}
        T \star S = a_T * a_S,\qquad a_{f \star S} = f * a_S
    \end{align}
    for $T, S \in \mathcal{S}^1$ and $f \in L^1(\R^{2d})$.

    \subsubsection{Cohen's class as operator-operator convolutions}
    The class of quadratic time-frequency distributions discussed in Section \ref{sec:tf_cohen} has a convenient formulation in quantum harmonic analysis using the Weyl quantization relations \eqref{eq:weyl_quantization_conv_props} above. By letting $\check{S}$ be the Weyl quantization of the tempered distribution $\Phi$ defining $Q_\Phi$ and using that $a_{\psi \otimes \psi} = W(\psi)$, we get that
    \begin{align}\label{eq:cohen_definition}
        Q_\Phi(\psi) = (\psi \otimes \psi) \star \check{S} =: Q_S(\psi).
    \end{align}
    This point of view makes it particularly easy to deduce properties of Cohen's class distributions such as bounding $L^p$ norms or characterizing positivity.

    \subsection{Functional analytic and probabilistic aspects of white noise}\label{sec:white_noise}
    The core of our approach to symbol recovery using white noise is computing spectrograms of random noise. This is inspired by recent theoretical work in \cite{Bardenet2021, Haimi2022, Ghosh2022} and others. For further details we refer the reader to the discussion in \cite{Romero2022} as we follow their proof strategy. The two main result which we need are stated in \cite{Romero2022} and we also state them here for the sake of completeness. The first is a version of the Hanson-Wright inequality \cite{Rudelson2013, Adamczak2015}.
    \begin{theorem}[{\cite[Theorem 3.1]{Romero2022}}]\label{theorem:gaussian_workhorse}
        Let $X$ be an $m$-dimensional complex Gaussian random variable with $X \sim \mathcal{CN}(0, \Sigma)$ and $A \in \C^{m \times m}$ Hermitian. Then there exists an universal constant $C_{hw} > 0$ such that for every $t > 0$,
        \begin{align*}
            \mathbb{P}\big( |\langle AX, X \rangle - \mathbb{E}\{ \langle AX, X \rangle \}| > t \big) \leq 2\exp\left( -C_{hw}\min\left\{ \frac{t^2}{\Vert        \Sigma\Vert^2_s \Vert A \Vert_F^2}, \frac{t}{\Vert \Sigma \Vert_s \Vert A \Vert_s} \right\} \right)
        \end{align*}
        where $\Vert \cdot \Vert_{s}$ and $\Vert \cdot \Vert_{F}$ are the spectral and Frobenius norms, respectively.
    \end{theorem}
    
    Secondly, the next lemma gives a constructive way to deal with the application of an operator to white noise.
    \begin{lemma}[{\cite[Lemma 4.2]{Romero2022}}]\label{lemma:random_expansion}
        Let $g$ be a Schwartz function with $\Vert g \Vert_{L^2} = 1$, $f \in L^1(\R^{2d})$ and $\mathcal{N}$ a realization of complex white noise with variance $\sigma^2$. Then there exists a sequence $(\alpha_m)_m$ where $\alpha_m \sim \mathcal{CN}(0, \sigma^2)$ of independent complex normal variables such that almost surely,
        $$
        A_f^g(\mathcal{N}) = \sum_m \lambda_m \alpha_m h_m\qquad \text{where } A_f^g = \sum_m \lambda_m (h_m \otimes h_m)
        $$
        with almost sure absolute convergence in $L^2(\R^{2d})$.
    \end{lemma}
    We also remark that if our white noise is not complex but rather real valued, all results will still hold but with possibly larger constants. See \cite[Section 2.2]{Romero2022} for a discussion on this.
    
    \subsection{Approximate identities}\label{sec:approx_id}
    The \emph{variation} of a function $f \in L^1(\R^{2d})$ is defined as
    $$
        \Var(f) = \sup \left\{ \int_{\R^{2d}} f(z) \operatorname{div}\phi(z)\,dz : \phi \in C^1_c(\R^{2d}, \R^{2d}),\, \Vert \phi\Vert_\infty \leq 1 \right\}
    $$
    and in the special case where $f \in C^1(\R^{2d})$, it can be written as
    $$
        \Var(f) = \int_{\R^{2d}} |\nabla f (z)| \,dz.
    $$
    We say that functions $f$ with $\Var(f) < \infty$ have \emph{bounded variation}. In the case where $f$ is the indicator function of some compact subset $\Omega \subset \R^{2d}$ with smooth boundary, the variation of $f$ is equal to the Haussdorff measure of the boundary $\partial \Omega$ \cite{Evans1991-hm}.
	
    In what follows, we will want to measure how much a function is changed when it is convolved with some kernel. The next lemma quantifies this using the concept of variation introduced above.
    \begin{lemma}[{\cite[Lemma 3.2]{Abreu2015}}]\label{lemma:bounded_variation_approximation_error}
        Let $\psi \in L^1(\R^{2d})$ have bounded variation and $\phi \in L^1(\R^{2d})$ with $\int_{\R^{2d}} \phi(z)\,dz = 1$, then
        $$
        \Vert \psi * \phi - \psi \Vert_{L^1} \leq \Var(\psi) \int_{\R^{2d}} |z| |\phi(z)|\,dz.
        $$
    \end{lemma}
    In the following, we will sometimes refer to $\phi$ as the \emph{blurring kernel}.

    \section{Recovery via white noise}\label{sec:white_noise_recovery}
    The idea behind symbol recovery via white noise is easy to summarize; on average, the spectrograms of white noise are constant so by looking at filtered white noise, the spectrograms reveal the characteristics of the filter. As spectrograms are inherently quadratic, this essentially limits us to real valued non-negative symbols as we can only estimate the squared modulus $|f|^2$.

    \subsection{Preliminary results}
    Before proceeding with a proof of Theorem \ref{theorem:L1_theorem}, we collect some preliminary results in the following proposition which is stated under looser conditions than Theorem \ref{theorem:L1_theorem} and details all the intermediate steps from the average observed spectrogram $\rho(z) = \frac{1}{K}\sum_{k=1}^K |V_\varphi(A_f^g \mathcal{N}_k)(z)|^2$ to $f^2$.
    \begin{proposition}\label{prop:white_noise_intermediate}    
        Let $f$ be a real-valued, bounded, integrable function with bounded derivative and bounded variation, $\rho$ the average observed spectrogram \eqref{eq:rho_def} with white noise variance $\sigma^2$ and $g, \varphi \in \mathcal{S}(\R^d)$ with $\Vert g \Vert_{L^2} = \Vert \varphi \Vert_{L^2} = 1$. Then there exists a constant $C > 1$  such that
        \begin{align}\label{eq:main_1}
            \mathbb{P}\left(\left|\frac{\rho(z)}{\sigma^2} - \vartheta(z) \right| > t\right) &\leq 3 \exp\left( -C K \min\left( \frac{t^2}{\vartheta(z)^2}, \frac{t}{\vartheta(z)} \right) \right),\\\label{eq:main_2}
            \big\Vert \vartheta - f^2 * |V_\varphi g|^2\big\Vert_{L^\infty} &\leq \Vert f \Vert_{L^\infty}\left( \sum_{|\alpha| = 1} \Vert \partial^\alpha f \Vert_{L^\infty} \right) \Vert g \Vert^4_{M^1},\\\label{eq:main_3}
            \big\Vert f^2 * |V_\varphi g|^2 - f^2 \big\Vert_{L^1} &\leq \left(\int_{\R^{2d}} \big|(\nabla f^2)(z)\big|\,dz\right) \left(\int_{\R^{2d}} |z| |V_\varphi g(z)|^2\,dz\right).
        \end{align}
    \end{proposition}
    For the first part of the proposition, we will need a version of \cite[Lemma 5.1]{Romero2022} with unknown variance and non-binary symbols. This requires minimal modifications to the original proof in \cite{Romero2022} and so we leave out the proof in the interest of brevity.
    \begin{lemma}\label{lemma:prob_geq_t}
        Let $f, \rho, \sigma, g, \varphi$ and $\vartheta$ be as in Proposition \ref{prop:white_noise_intermediate}. Then there exists $C > 0$ such that for every $z \in \R^{2d}$,
        $$
        \mathbb{P}\left(\left|\frac{\rho(z)}{\sigma^2} - \vartheta(z) \right| > t\right) \leq 3 \exp\left( -CK\min\left( \frac{t^2}{\vartheta(z)^2}, \frac{t}{\vartheta(z)} \right) \right).
        $$
    \end{lemma}
    We can now proceed with the proof of the proposition.    
    \begin{proof}[Proof of Proposition \ref{prop:white_noise_intermediate}]
        The first estimate of the proposition is precisely Lemma \ref{lemma:prob_geq_t} as stated above.
        
        For \eqref{eq:main_2}, we first claim that
        \begin{align}\label{eq:vartheta_expression}
            \vartheta(z) = \big(A_f^g \big)^2 \star (\varphi \otimes \varphi)\check{\,}(z).
        \end{align}
        Indeed, as $A_f^g = \sum_m \lambda_m (h_m \otimes h_m) \in \mathcal{S}^1$, it follows that $\big(A_f^g\big)^2 = \sum_m \lambda_m^2 (h_m \otimes h_m)$. Hence
        \begin{equation}\label{eq:vartheta_alternative}
    		\begin{aligned}
    			\big(A_f^g\big)^2 \star (\varphi \otimes \varphi)\check{\,}(z) &= \sum_m \lambda_m^2 (h_m \otimes h_m) \star (\varphi \otimes \varphi)\check{\,}(z)\\
                &= \sum_m \lambda_m^2 |V_\varphi h_m(z)|^2 = \vartheta(z)
    		\end{aligned}
    	\end{equation}
        by Example \ref{example:qha_rank_one_conv}. To proceed from here, we need to relate the product $(A_f^g)^2$ to $A_{f^2}^g$, a problem which has been studied by Cordero, Rodino and Gröchenig among others. In our situation, their results in \cite{Cordero2005}, \cite[Theorem 4 (i)]{Cordero2006} and \cite[Lemma 5]{Cordero2006} can be summarized as follows:
        \begin{align}\label{eq:product_formula_error}
            A_f^g A_f^g = A_{f^2}^g + V_g^* T V_g
        \end{align}
        where $T$ is an integral operator with kernel $G : \R^{2d} \times \R^{2d} \to \C$ given by
        \begin{align*}
            G(y,z) = f(y) \sum_{|\alpha| = 1} \int_0^1 (1-t) \partial^\alpha f(y+t(z-y))\,dt \frac{(z-y)^\alpha}{\alpha!} \langle \pi(z) g, \pi(y)g \rangle.
        \end{align*}
        Moreover, the norm of $V_g^* T V_g$ can be bounded as
        \begin{align}\label{eq:loc_op_prod_error_Linfty}
            \Vert V_g^* T V_g \Vert_{\mathcal{L}(L^2)} \leq \Vert a \Vert_{L^\infty} \left( \sum_{|\alpha| = 1} \Vert \partial^\alpha f \Vert_{L^\infty} \right) \Vert g \Vert_{M^1}^4.
        \end{align}
        Applying this result to $\big( A_f^g \big)^2$ yields
        \begin{align*}
            \big(A_f^g\big)^2 = A_{f^2}^g + V_g^* T V_g = f^2 \star (g \otimes g) + V_g^* T V_g.
        \end{align*}
        Plugging this into \eqref{eq:vartheta_alternative} and applying Example \ref{example:qha_rank_one_conv} yields
        \begin{align*}
            \vartheta(z) &= \big(f^2 \star (g \otimes g) + V_g^* T V_g \big) \star (\varphi \otimes \varphi)\check{\,}(z)\\
            &= f^2 * |V_\varphi g|^2(z) + (V_g^* T V_g) \star (\varphi \otimes \varphi)\check{\,}(z).
        \end{align*}
        Rearranging the above and applying Proposition \ref{prop:op_conv_properties} \ref{item:op_op_bound} together with \eqref{eq:loc_op_prod_error_Linfty}, we get the estimate
        \begin{align*}
            \big\Vert \vartheta - f^2 * |V_\varphi g|^2\big\Vert_{L^\infty} &= \Vert (V_g^* T V_g) \star (\varphi \otimes \varphi)\check{\,} \Vert_{L^\infty} \leq \Vert V_g^* T V_g \Vert_{\mathcal{L}(L^2)} \Vert \varphi \Vert_{L^2}^2\\
            &\leq \Vert f \Vert_{L^\infty} \left( \sum_{|\alpha| = 1} \Vert \partial^\alpha f \Vert_{L^\infty} \right) \Vert g \Vert_{M^1}^4.
        \end{align*}
        Lastly for \eqref{eq:main_3}, applying Lemma \ref{lemma:bounded_variation_approximation_error} with $\psi = f^2$ and $\phi = |V_\varphi g|^2$ yields the desired conclusion.
    \end{proof}
    \subsection{Proof of Theorem \ref{theorem:L1_theorem}}
    For Theorem \ref{theorem:L1_theorem}, much of the machinery from the proof of Proposition \ref{prop:white_noise_intermediate} can be reused but we will need an additional estimate on the localization operator product asymptotics and a lemma turning the estimate in Lemma \ref{lemma:prob_geq_t} into an $L^1$ error which is similar to \cite[Lemma 5.4]{Romero2022}.

    As a first step, we state a simplified version of \cite[Theorem 2]{Stinespring1958} adapted to a context in which we will soon need it.
    \begin{lemma}\label{lemma:stinespring_trace_bound}
        Let $T : L^2(\R^{2d}) \to L^2(\R^{2d})$ be an integral operator with kernel $G : \R^{2d} \times \R^{2d} \to \C$ that has compact support in the first variable. Then
        $$
        \Vert T \Vert_{\mathcal{S}^1} \leq A\left[ \Vert G \Vert_{L^2}  + \left(\sum_{j=1}^{2d} \big\Vert \partial_j^{d+1} G \big\Vert_{L^2}^2 \right)^{1/2} \right]
        $$
        where the constant $A$ is independent of $G$.
    \end{lemma}
    Armed with this lemma, we can bound the trace norm of the $V_\varphi^* T V_\varphi$ error operator from \eqref{eq:product_formula_error} above.
    \begin{lemma}\label{lemma:E_1S_1_bound}
        Let $f \in C^{d+2}_c(\R^{2d})$ and $g \in \mathcal{S}(\R^{d})$ with $\Vert g \Vert_{L^2} = 1$, then there exists a constant $A$ independent of $f$ and $g$ such that
        $$
        \Vert V_g^* T V_g \Vert_{\mathcal{S}^1} \leq A\left[ \Vert G \Vert_{L^2}  + \left(\sum_{j=1}^{2d} \big\Vert \partial_j^{d+1} G \big\Vert_{L^2}^2 \right)^{1/2} \right] < \infty
        $$
        where
        $$
        G(y,z) = f(y) \left( \sum_{|\alpha| = 1} \int_0^1 \partial^\alpha f(y+t(z-y))\,dt (z-y) \right) V_g g(y-z).
        $$
    \end{lemma}
    \begin{proof}
        Since $V_g$ is an isometry and $\Vert SR \Vert_{\mathcal{S}^1} \leq \Vert S \Vert_{\mathcal{S}^1} \Vert R \Vert_{\mathcal{L}(L^2)}$ for operators $S$ and $R$, we conclude that it suffices to bound the trace norm of the integral operator $T$. The bound in the formulation follows directly upon applying Lemma \ref{lemma:stinespring_trace_bound}.
        
        The finiteness of the error bound follows from the support of $f$ being compact and that $g \in \mathcal{S}(\R^d)$ via \cite[Theorem 11.2.5]{grochenig_book}.
    \end{proof}
    Lastly we formulate the promised $L^1$-estimate, based on Lemma \ref{lemma:prob_geq_t}. Its formulation and proof is similar to that of \cite[Lemma 5.4]{Romero2022}.
    \begin{lemma}\label{lemma:prob_to_int_estimate}
        Let $f \in L^1(\R^{2d}) \cap L^\infty(\R^{2d})$. Then
        \begin{align*}
            \mathbb{P}\left( \int_{\R^{2d}} \left| \frac{\rho(z)}{\sigma^2} - \vartheta(z) \right|dz \geq \gamma \right) \leq \frac{\Vert f \Vert_{L^2}^2}{\gamma \sqrt{K}} \left[ \frac{3 \sqrt{\pi} }{2\sqrt{C}}\operatorname{erf}(\sqrt{CK}) + \frac{3}{ \sqrt{C} } e^{- \sqrt{CK} } \right].
        \end{align*}
    \end{lemma}
    \begin{proof}
        By Markov's inequality applied to the random variable $\int_{\R^{2d}} \left| \frac{\rho(z)}{\sigma^2} - \vartheta(z) \right|dz$, we have
        \begin{align}\nonumber
            \mathbb{P}\left( \int_{\R^{2d}} \left| \frac{\rho(z)}{\sigma^2} - \vartheta(z) \right|\,dz \geq \gamma \right) &\leq \frac{1}{\gamma}\mathbb{E}\left\{ \int_{\R^{2d}} \left| \frac{\rho(z)}{\sigma^2} - \vartheta(z) \right|\,dz \right\}\\\nonumber
            &=\frac{1}{\gamma} \int_{\R^{2d}}\mathbb{E}\left\{ \left| \frac{\rho(z)}{\sigma^2} - \vartheta(z)\right| \right\}\,dz\\\label{eq:markov_deep}
            &=\frac{1}{\gamma} \int_{\R^{2d}} \int_{0}^\infty \mathbb{P}\left\{ \left| \frac{\rho(z)}{\sigma^2} - \vartheta(z)\right| \geq t \right\}\,dt\,dz.
        \end{align}
        Next we estimate the inner integral for each $z \in \R^{2d}$ using Lemma \ref{lemma:prob_geq_t} as
        \begin{align*}
            \int_{0}^\infty \mathbb{P}\left\{ \left| \frac{\rho(z)}{\sigma^2} - \vartheta(z)\right| \geq t \right\}\,dt &\leq 3 \int_0^\infty \exp\left( -CK \min\left( \frac{t^2}{\vartheta(z)^2}, \frac{t}{\vartheta(z)} \right) \right)\,dt\\
            &= 3 \int_0^{\vartheta(z)} \exp\left( -\frac{CK t^2}{\vartheta(z)^2} \right)\,dt + 3 \int_{\vartheta(z)}^\infty \exp\left( -\frac{CK t}{\vartheta(z)} \right)\,dt\\
            &= \vartheta(z)\left[ \frac{3\sqrt{\pi}}{2\sqrt{CK}}\operatorname{erf}\big(\sqrt{CK}\big) + \frac{3}{CK} e^{-CK} \right].
        \end{align*}
        When computing the integral of this over $\R^{2d}$ we will need to compute the $L^1$-norm of $\vartheta$. By \eqref{eq:vartheta_expression} and Proposition \ref{prop:op_conv_properties} \ref{item:op_op_bound} with $p=1$, it can be bounded as
        \begin{align*}
            \Vert \vartheta \Vert_{L^1} &= \big\Vert (A_f^g)^2 \star (\varphi \otimes \varphi) \big\Vert_{L^1}\\
            &\leq \big\Vert (A_f^g)^2 \big\Vert_{\mathcal{S}^1} \Vert \varphi \otimes \varphi \big\Vert_{\mathcal{S}^1} = \Vert A_f^g \Vert_{\mathcal{S}^2}^2 \leq \Vert f \Vert_{L^2}^2 \Vert \varphi \otimes \varphi \Vert_{\mathcal{S}^1}^2 = \Vert f \Vert_{L^2}^2
        \end{align*}
        where we used Proposition \ref{prop:op_conv_properties} \ref{item:func_op_bound} with $p=2$ for the second to last step. Plugging this back into \eqref{eq:markov_deep} yields
        \begin{align*}
            \mathbb{P}\left( \int_{\R^{2d}} \left| \frac{\rho(z)}{\sigma^2} - \vartheta(z) \right|\,dz \geq \gamma \right) &\leq\frac{1}{\gamma} \Vert \vartheta \Vert_{L^1} \left[ \frac{3 \sqrt{\pi} }{2\sqrt{CK}}\operatorname{erf}\big(\sqrt{CK}\big) + \frac{3}{ \sqrt{CK} } e^{-\sqrt{CK} } \right]\\
            &\leq \frac{\Vert f \Vert_{L^2}^2}{\gamma \sqrt{K} } \left[ \frac{3 \sqrt{\pi} }{2\sqrt{C}}\operatorname{erf}\big(\sqrt{CK}\big) + \frac{3}{C\sqrt{K} } e^{- CK } \right]
        \end{align*}
        as desired.
    \end{proof}
    We are now ready to complete the proof of Theorem \ref{theorem:L1_theorem}.
    \begin{proof}[Proof of Theorem \ref{theorem:L1_theorem}]
        We first claim that
        \begin{align*}
            \big\Vert \vartheta - f^2 * |V_\varphi g|^2 \big\Vert_{L^1} \leq A\left[ \Vert G \Vert_{L^2}  + \left(\sum_{j=1}^{2d} \Vert \partial_j^{d+1} G \Vert_{L^2}^2 \right)^{1/2} \right].
        \end{align*}
        Indeed, this follows from Lemma \ref{lemma:E_1S_1_bound} as
        \begin{align*}
            \big\Vert \vartheta - f^2 * |V_\varphi g|^2 \big\Vert_{L^1} &= \big\Vert \big(A_f^g\big)^2 \star (\varphi \otimes \varphi)\check{\,} - A_{f^2}^g \star (\varphi \otimes \varphi)\check{\,}\big\Vert_{L^1}\\
            &= \Vert (V_g^* T V_g) \star (\varphi \otimes \varphi)\check{\,} \Vert_{L^1}\\
            &\leq \Vert V_g^* T V_g \Vert_{\mathcal{S}^1} \Vert \varphi \otimes \varphi \Vert_{\mathcal{S}^1} \\
            &\leq A\left[ \Vert G \Vert_{L^2}  + \left(\sum_{j=1}^{2d} \big\Vert \partial_j^{d+1} G \big\Vert_{L^2}^2 \right)^{1/2} \right]
        \end{align*}
        where we used Proposition \ref{prop:op_conv_properties} \ref{item:op_op_bound} for the second to last step.

        We now expand the left hand side in the $\left\Vert\frac{\rho}{\sigma^2} - f^2 \right\Vert_{L^1} > A B_1 + B_2 + t$ inequality using the above and Lemma \ref{lemma:bounded_variation_approximation_error} with $\psi = f^2$ and $\phi = |V_\varphi g|^2$ to find
        \begin{align*}
            \mathbb{P}&\left( \left\Vert \frac{\rho}{\sigma^2} - f^2 \right\Vert_{L^1} > B_1 + B_2 + t \right) \\
            &\hspace{2mm}\leq \mathbb{P}\left( \left\Vert \frac{\rho}{\sigma^2} - \vartheta \right\Vert_{L^1} + \big\Vert \vartheta - f^2 * |V_\varphi g|^2 \big\Vert_{L^1} + \big\Vert f^2 * |V_\varphi g|^2 - f^2 \big\Vert_{L^1} > B_1 + B_2 + t \right)\\
            &\hspace{2mm}\leq \mathbb{P}\left( \left\Vert \frac{\rho}{\sigma^2} - \vartheta \right\Vert_{L^1} + B_1 + B_2 > B_1 + B_2 + t \right)\\
            &\hspace{2mm}= \mathbb{P}\left( \left\Vert \frac{\rho}{\sigma^2} - \vartheta \right\Vert_{L^1} > t \right) \leq \frac{\Vert f \Vert_{L^2}^2}{t \sqrt{K} } \left[ \frac{3 \sqrt{\pi} }{2\sqrt{C}}\operatorname{erf}(\sqrt{CK}) + \frac{3}{C \sqrt{K} } e^{- CK } \right]
        \end{align*}
        where we in the last step used Lemma \ref{lemma:prob_to_int_estimate}.
    \end{proof}
    
    \begin{remark}
        Both Proposition \ref{prop:white_noise_intermediate} and Theorem \ref{theorem:L1_theorem} have clear analogues in the Cohen's class case which we believe to hold true. Indeed, it is straight-forward to show that
        $$
        \mathbb{E}\left( \frac{1}{K} \sum_{k=1}^K Q_S(f \star S(\mathcal{N}_k)) \right) \xrightarrow[K \to \infty]{} \sum_{m=1}^\infty \lambda_m^2 Q_S(h_m)(z),
        $$
        but controlling the error estimates requires generalizing Lemma \ref{lemma:prob_geq_t} to the non rank-one case which is considerably more difficult.
    \end{remark}
    
    \begin{remark}
        The quantity $\int_{\R^{2d}} |z||V_\varphi g(z)|^2\,dz$ which appears in Proposition \ref{prop:white_noise_intermediate} and Theorem \ref{theorem:L1_theorem} should be seen as punishing the case $\varphi \neq g$, i.e., the reconstruction window differing from the window function $g$.
    \end{remark}

    \section{Recovery via spectral data}\label{sec:spectral_recovery}
    In this section we discuss and prove the two recovery results Theorem \ref{theorem:accumulated} and Theorem \ref{theorem:fourier_deconvolution} which are dependent on the eigenvalues and eigenfunctions of the localization operator. As we will see in Section \ref{sec:numerical_implementation}, the instability of eigenvalues and eigenfunctions can cause issues for these approaches but they still perform well.
    
    \subsection{Weighted accumulated Cohen's class}\label{sec:weighted_accumulated_cohen}
    The accumulated Cohen's class, introduced in \cite{Luef2019_acc}, is a generalization of accumulated spectrograms from \cite{Abreu2015} where it was used for symbol recovery for binary localization operators $A^g_{\chi_\Omega}$. There, a central idea was that the eigenvalues can be separated into two groups with the first $\approx \lceil|\Omega|\rceil$ being close to $1$, followed by a sharp ``plunge region'' after which the remaining eigenvalues are all close to $0$. This fact was originally proved in \cite{Feichtinger2001}. Based on this fact, the quantity
    \begin{align}\label{eq:accumulated_spectrogram}
        \sum_{m=1}^{\lceil |\Omega| \rceil} |V_g(h_m)(z)|^2 \approx \chi_\Omega(z)
    \end{align}
    was defined as the \emph{accumulated spectrogram}. Later on, \cite{Luef2019_acc} extended the concept to mixed-state localization operators $f \star S$ by replacing the spectrograms in \eqref{eq:accumulated_spectrogram} by Cohen's class distributions by approaching the proof from a quantum harmonic analysis perspective.
    
    Much of the work in these papers is focused on showing that the accumulated Cohen's class \eqref{eq:accumulated_spectrogram} is close to the quantity
    \begin{align}\label{eq:weighted_accumulated_cohen}
        \sum_m \lambda_m Q_S(h_m)(z)
    \end{align}
    by going into specifics on the decay of the eigenvalues. However, since computing the accumulated spectrogram already requires knowing the eigenfunctions, we (almost always) have exact knowledge of the eigenvalues and can bypass this approximation step and include the eigenvalues in the estimator. In this way, the error of the approximation can be decreased with no loss in performance or increase in runtime. Moreover, we do not require a priori knowledge of $|\Omega|$ to decide the number of eigenpairs to include. A consequence of this approach is that the resulting estimator also works well for non-binary localization operators whose eigenvalues do not follow the same $0-1$ dichotomy. We refer to the quantity \eqref{eq:weighted_accumulated_cohen} as the \emph{weighted} accumulated Cohen's class to highlight the addition of the eigenvalue weights.
    
    Both \cite{Abreu2015} and \cite{Luef2019_acc} restricted their attention to the case where the window $g$ or the operator window $S$ was known a priori. We lift this restriction by introducing a reconstruction window $\varphi$ or reconstruction operator window $T$ which does not have to agree with the original window $g$ or $S$ in the same way as we did for the average observed spectrogram. As we will see in the proof below, the proper estimator then instead becomes $\sum_m \lambda_m Q_T(h_m)(z)$.
    
    \begin{proof}[Proof of Theorem \ref{theorem:accumulated}]
        The key observation for the proof is that $\sum_m \lambda_m Q_T(h_m) = f * (S \star \check{T})$. To see this, expand $f \star S$ in its singular value decomposition $f \star S = \sum_m \lambda_m (h_m \otimes h_m)$ and note that
        \begin{align*}
            f * (S \star \check{T}) &= (f \star S) \star \check{T} = \left( \sum_m \lambda_m (h_m \otimes h_m) \right) \star \check{T}\\
            &= \sum_m \lambda_m (h_m \otimes h_m) \star \check{T} = \sum_m \lambda_m Q_T(h_m)
        \end{align*}
        where we used \eqref{eq:cohen_definition} for the last step. We can now compute
        \begin{align*}
            \left\Vert \sum_{m=1}^N \lambda_m Q_T(h_m) - f \right\Vert_{L^1} &\leq \left\Vert\sum_{m=1}^N \lambda_m Q_T(h_m) - \sum_{m=1}^\infty \lambda_m Q_T(h_m)\right\Vert_{L^1} + \big\Vert f * (S \star \check{T}) - f \big\Vert_{L^1}\\
            &\leq \sum_{m = N+1}^\infty |\lambda_m| \Vert Q_T(h_m)\Vert_{L^1} + \big\Vert f * (S \star \check{T}) -f \big\Vert_{L^1}\\
            &\leq \sum_{m=N+1}^\infty |\lambda_m| + \Var(f)\int_{\R^{2d}}|z|(S \star \check{T})(z)\,dz
        \end{align*}
        where we used that $\Vert Q_T(h_m)\Vert_{L^1} \leq 1$ by Proposition \ref{prop:op_conv_properties} \ref{item:op_op_bound} with $p = 1$ and the estimate in Lemma \ref{lemma:bounded_variation_approximation_error}.
    \end{proof}
    \begin{remark}
        Ideally, we would want $S \star \check{T}$ to be a Dirac delta to make the above reconstruction exact in the sense that $\sum_m \lambda_m Q_T(h_m) = f$. The closest we can get to this is in the lattice setting where such a construction is possible which is discussed in \cite[Section 6.1]{Skrettingland2020}. The error incurred from $S \star \check{T} \neq \delta_0$ is partially captured in the $\int_{\R^{2d}}|z|(S \star \check{T})(z)\,dz$ factor which simplifies to $\int_{\R^{2d}}|z||V_\varphi g(z)|^2\,dz$ in the rank-one setting which we recognize from Section \ref{sec:white_noise_recovery}. The specifics of this error were discussed and illustrated in \cite[Figure 3]{Abreu2015}.
    \end{remark}
    
    The reader familiar with \cite{Luef2019_acc} will note that we essentially followed the exact same path for the proof as in that paper without restricting ourselves to indicator functions $f = \chi_\Omega$ and allowing $T \neq S$.
    
    The recovery procedure detailed above is clearly linear and hence it is easy to see that it is continuous on $\mathcal{S}^1$. We mean this in the sense that if $I$ is the map $f \star S \mapsto f * (S \star \check{T})$ and $A \in \mathcal{S}^1$ is a perturbation, then
    \begin{align}\label{eq:accumulated_continuous}
        \big\Vert I(f \star S + \varepsilon A) - I(f \star S) \big\Vert_{L^1} = \varepsilon \Vert A \star \check{T} \Vert_{L^1} \leq \varepsilon \Vert A \Vert_{\mathcal{S}^1}
    \end{align}
    by linearity and Proposition \ref{prop:op_conv_properties} \ref{item:op_op_bound}.

    As the estimator converges to a convolution in the $N\to \infty$ case, we can attempt to perform a deconvolution procedure to recover $f$ exactly if we know the blurring kernel and its Fourier transform is zero-free. This is investigated numerically in Section \ref{sec:num_acc_spec}.
    
    \subsection{Weighted accumulated Wigner distribution}
    The approach in Theorem \ref{theorem:fourier_deconvolution} is perhaps the simplest of those detailed in this paper once framed as just computing the Weyl symbol of the localization operator and comparing with $f$ (see \eqref{eq:weyl_quantization_conv_props}). There is also no requirement for a reconstruction window in this situation as our construction only depends on the spectral data of the localization operator. Note that we again adopt the \emph{weighted} terminology to highlight the eigenvalue dependence as in Section \ref{sec:weighted_accumulated_cohen}.
    
    \begin{proof}[Proof of Theorem \ref{theorem:fourier_deconvolution}]
        We prove the full case where $S$ is a positive trace-class operator and note that the special rank-one case follows from it.
        
        As discussed in Section \ref{sec:weyl_quantization}, the Weyl symbol of the function-operator convolution $f \star S$ is given by $ f * a_S$ where $a_S$ is the Weyl symbol of $S$. By the linearity of the Weyl symbol mapping $S \mapsto a_S$, we can compute this by using the spectral decomposition of $S = \sum_n s_n (g_n \otimes g_n)$ and the fact that $a_{g \otimes g} = W(g)$:
        \begin{align*}
            \sum_m \lambda_m W(h_m) = a_{f \star S} = f * a_S = f * \sum_n s_n W(g_n).
        \end{align*}
        In order for the sum in the left-hand side to converge in $L^1$, we need for the Wigner distribution of each eigenfunction $h_m$ to be integrable. This is equivalent to $h_m \in M^1(\R^{d})$ which follows from $g_n \in \mathcal{S}(\R^d)$ by \cite[Theorem 4.1]{Bastianoni2020}.
        
        The $L^1$-error estimate now follows by applying Lemma \ref{lemma:bounded_variation_approximation_error} with $\psi = f$ and blurring kernel $\phi = \sum_n s_n W(\varphi_n)$.
    \end{proof}
    Just as in Theorem \ref{theorem:accumulated}, we can possibly deconvolve $\sum_m \lambda_m W(h_m) = f * \sum_n s_n W(g_n)$ to recover $f$ exactly provided the Fourier transform $\mathcal{F}\left(\sum_n s_n W(g_n)\right)$ is zero-free. It turns out that yields precisely the same expression for the deconvolution as if we would naively deconvolve $A_f^g = f \star (g \otimes g)$ using the Fourier-Wigner transform from Section \ref{sec:fourier_wigner_deconv}.
    
    Note that the sum $\sum_m \lambda_m W(h_m)$ is easily seen to converge pointwise by the bound $|W(h_m)(z)| \leq 2^d \Vert h_m \Vert_{L^2}^2$ while we need the extra condition on the window for $L^1$-convergence. This is why we could not formulate Theorem \ref{theorem:fourier_deconvolution} with partial sums as we did for Theorem \ref{theorem:accumulated}.
    
    The above argument can be taken another step to show that the reconstruction procedure is not stable as was the case for accumulated spectrograms as shown in \eqref{eq:accumulated_continuous}. To see that the inverse mapping $I : \mathcal{S}^1 \to L^1(\R^{2d}),\, f \star S \mapsto \sum_m \lambda_m W(h_m)$ is not continuous, fix $\psi \in L^2(\R^d) \setminus M^1(\R^d)$ and consider the perturbation operator $A = \psi \otimes \psi \in \mathcal{S}^1$ for which we have
    $$
    \Vert I(f \star S + \varepsilon A) - I(f \star S) \Vert_{L^1} = \varepsilon \Vert I(\psi \otimes \psi) \Vert_{L^1} = \varepsilon \Vert W(\psi) \Vert_{L^1} = \infty.
    $$
    In Section \ref{sec:num_acc_wig} we provide an example showing the performance of the estimator $\sum_m \lambda_m W(h_m)$ and discuss aspects of the numerical implementation.

    \section{Recovery via plane tiling}\label{sec:plane_tiling}
    Using quantum harmonic analysis, it is easy to show that the spectrograms of an orthonormal basis add up to the function which is identically $1$. Indeed, using the relation $1 \star (\varphi \otimes \varphi) = \Vert \varphi \Vert_{L^2}^2 I_{L^2}$ from \cite[Proposition 3.2 (3)]{Werner1984}, we get for a normalized $\varphi \in L^2(\R^d)$ that
    \begin{align*}
        \sum_n |V_\varphi (e_n)|^2 &= \sum_n (e_n \otimes e_n) \star (\varphi \otimes \varphi)\check{\,}\\
        &= \left( \sum_n (e_n \otimes e_n) \right) \star (\varphi \otimes \varphi)\check{\,}\\
        &= I \star (\varphi \otimes \varphi)\\
        &= 1 * (\varphi \otimes \varphi) \star (\varphi \otimes \varphi)\check{\,} = 1 * |V_\varphi \varphi|^2 = \Vert V_\varphi \varphi \Vert_{L^2}^2 = 1.
    \end{align*}
    Intuitively, we should expect that those basis elements whose spectrograms are primarily supported away from the support of $f$ should lose most of their mass when we apply $A_f^g$ to them and the rest should remain intact or be scaled by something proportional to $f$. This is the motivation for the plane tiling approach which we prove below. The proof is rather straight-forward and we are able to inherit the main error estimate from Theorem \ref{theorem:L1_theorem} as the sum approaches the same quantity $\vartheta$ from \eqref{eq:vartheta} in the $K\to \infty$ situation in that theorem.
    \begin{proof}[Proof of Theorem \ref{theorem:plane_tiling}]
        We first rework the estimator $\sum_n |V_\varphi(A_f^g e_n)(z)|^2$ into a more manageable form using the self-adjointness of $A_f^g$ and Example \ref{example:qha_rank_one_conv} as
        \begin{align*}
            \sum_n |V_\varphi(A_f^g e_n)(z)|^2 &= \sum_n \big( A_f^g e_n \otimes A_f^g e_n \big) \star (\varphi \otimes \varphi)\check{\,}(z)\\
            &= A_f^g\left(\sum_n e_n \otimes e_n\right)A_f^g \star (\varphi \otimes \varphi)\check{\,}(z)\\
            &= \big( A_f^g I A_f^g \big) \star (\varphi \otimes \varphi)\check{\,}(z)\\
            &= \big( A_f^g \big)^2 \star (\varphi \otimes \varphi)\check{\,}(z) = \vartheta(z)
        \end{align*}
        where $\vartheta$ is the same as in Section \ref{sec:white_noise_recovery}. The same analysis on the size of $\Vert \vartheta - f^2 \Vert_{L^1}$ from the proof of Theorem \ref{theorem:L1_theorem} again applies and yields the desired conclusion.
    \end{proof}
    \begin{remark}
        The above result can be extended to mixed-state localization operators as was done in Section \ref{sec:spectral_recovery} through some technical considerations. More specifically, it is possible to control the error $\Vert \sum_n Q_T((f \star S) e_n) - f^2 \Vert_{L^1}$ if $S$ and $T$ are positive rank-one operators whose spectral decomposition consists of Schwartz functions. This is done by bounding the trace norm of the error operator in the expansion $(f \star S)^2 = f^2 \star S + V_g^* T V_g$ in a similar way to how it was done in the proof of Theorem \ref{theorem:L1_theorem}.
    \end{remark}
    If one has a choice between this method and the average observed spectrogram from Section \ref{sec:white_noise_recovery}, the better method depends on the size of the support of the symbol. A good choice for the orthonormal basis are time-frequency shifted Hermite functions as these will tile out a growing circle \cite{daubechies1988_loc}. By choosing the time-frequency shift appropriately, we can thus approximate $f(z_0)$ well by $\sum_{n=1}^N |V_\varphi(A_f^g(\pi(z_0) h_n))(z)|^2$ for small $N$. This is illustrated in Section \ref{sec:plane_tiling_numerics} below.

    \section{Recovery via Gabor projection}\label{sec:gabor_projection}
    The main difference between Gabor projection and the other proposed methods is that we are estimating each pixel of $f$ in an independent way. Notably we do this without any $L^\infty$ continuity guarantees. The idea of the procedure laid out in the introduction is purely intuitive and we can strengthen this intuition by investigating the special case $\varphi = g$ in detail. In particular, this is when the procedure is a (pointwise) projection to the Gabor space $V_g(L^2)$ in the sense that
    \begin{align}\label{eq:gabor_projection_estimator}
        V_g(A_f^g(\pi(z)g))(z) = V_g V_g^* (f \cdot V_gg(\cdot-z))(z) = P_{V_g(L^2)}[f \cdot V_gg(\cdot-z)](z).
    \end{align}
    If we had some $L^\infty$ continuity conditions for the orthogonal projection onto the Gabor space we could therefore get more guarantees on the performance of this method. Note also that since $V_gg$ is continuous, we could use $V_g(A_f^g(\pi(z)g))$ to estimate the value of $f$ for a neighborhood of $z$ which would lead to a shorter runtime but worse recovery performance.

    \begin{proof}[Proof of Theorem \ref{theorem:gabor_projection}]
        The proof essentially boils down to expanding the short-time Fourier transform and synthesis $V_g^*$ in \eqref{eq:gabor_projection_estimator}, changing the order of integration, and identifying the blurring kernel. Indeed,
        \begin{align*}
            V_\varphi\big(A_f^g(\pi(z)g)\big)(z) &= \int_{\R^d} V_g^*\big(f \cdot V_g(\pi(z)\varphi)\big)(t) \overline{\pi(z)\varphi(t)}\,dt\\
            &=\int_{\R^d} \int_{\R^{2d}} f(w) V_g(\pi(z)\varphi)(w) \pi(w)g(t)\overline{\pi(z)\varphi(t)}\,dw\,dt.
        \end{align*}
        From here we may use Fubini as can be seen by considering
        \begin{align*}
        \int_{\R^d} \int_{\R^{2d}} &\big|f(w) V_g(\pi(z)\varphi)(w) \pi(w)g(t)\overline{\pi(z)\varphi(t)}\big|\,dw\,dt\\
        &= \int_{\R^{2d}} |f(w)| |V_g \varphi(w-z)| \underbrace{\int_{\R^d} \big| \pi(w)g(t)\overline{\pi(z)\varphi(t)} \big|\,dt}_{\leq \Vert g \Vert_{L^2} \Vert \varphi \Vert_{L^2}} \,dw\\
        &\leq \Vert g \Vert_{L^2} \Vert \varphi \Vert_{L^2} \Vert f \Vert_{L^1} \Vert V_g \varphi \Vert_{L^\infty} < \infty.
        \end{align*}
        Hence we may continue with a similar computation as
            \begin{align*}
            V_\varphi(A_f^g(\pi(z)g))(z) &= \int_{\R^{2d}} f(w) e^{-2\pi i x \cdot \eta} V_g\varphi(w-z) \left( \int_{\R^d} \pi(w)g(t)\overline{\pi(z)\varphi(t)} \,dt\right)\,dw\\
            &=\int_{\R^{2d}} f(w)  e^{-2\pi i x \cdot \eta}V_g\varphi(w-z) \overline{V_g(\pi(z) \varphi)(w)}\,dw\\
            &=\int_{\R^{2d}} f(w) |V_\varphi g(z-w)|^2\,dw = f * |V_\varphi g|^2(z)
        \end{align*}
        where we used that $V_g(\pi(z)\varphi)(w) = e^{-2\pi i x \cdot \eta} V_g\varphi(w-z) $ for $z = (x,\omega),\,w = (y, \eta)$ and cancelled out the two phase factors.
    \end{proof}
    Note in particular that the above construction handles noise in the input $\pi(z)\varphi$ gracefully in the sense that if we add a perturbation, the effect on the absolute value of the estimator is proportional to the $L^2$ energy of the noise as can be seen by the linearity of the estimator and the boundedness of $A_f^g$ and $V_\varphi$. 

    \section{Numerical implementation}\label{sec:numerical_implementation}
    In computer applications there is no continuum and the integral in the the localization operator definition \eqref{eq:loc_op_def} is replaced by a sum, most often over some lattice, yielding what is referred to as a \emph{Gabor multiplier} \cite{Feichtinger2003}. While showing that the results of this paper carry over to this setting is a non-trivial undertaking which we do not attempt, we settle for investigating the numerical behavior and draw only empirical conclusions. Do note however that many results on localization operators do carry over to the Gabor multiplier setting \cite{Cordero2006_prod, Feichtinger2003, Feichtinger2023_measure} and that in particular localization operators can be approximated in the $\mathcal{S}^1$ norm by Gabor multipliers \cite{Feichtinger2023_measure}.

    Before looking at details, we present a brief visual comparison between the different methods for a collection of symbols and lattice parameters. All of the figures presented in this section were generated using the Large Time/Frequency Analysis Toolbox (LTFAT) \cite{ltfatnote030} and the associated code can be found in the GitHub repository\footnote{\url{https://github.com/SimonHalvdansson/Localization-Operator-Symbol-Recovery}}.
    \begin{figure}[H]
        \centering
        \includegraphics[width=6in]{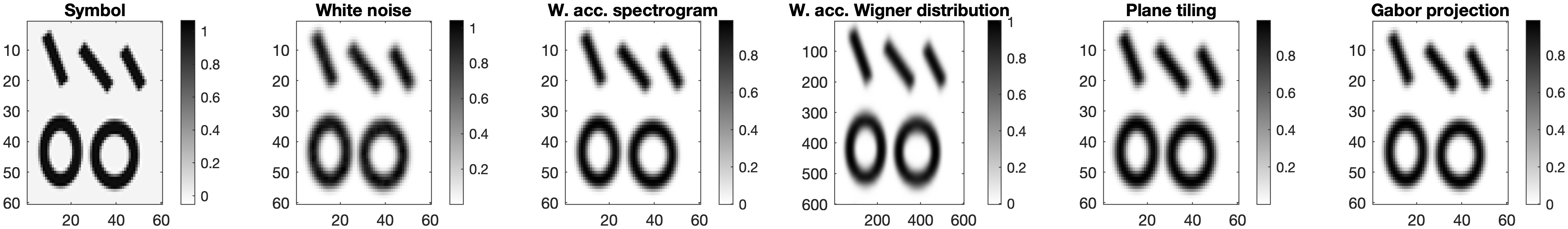}
        \vskip2mm
        \includegraphics[width=6in]{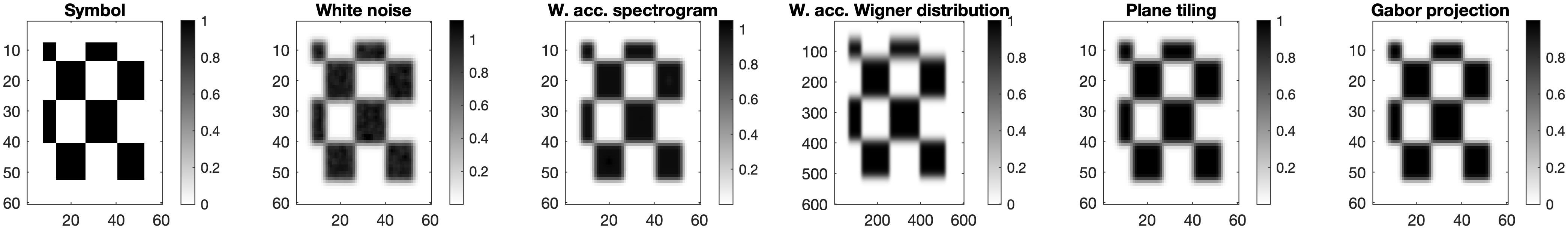}
        \vskip2mm
        \includegraphics[width=6in]{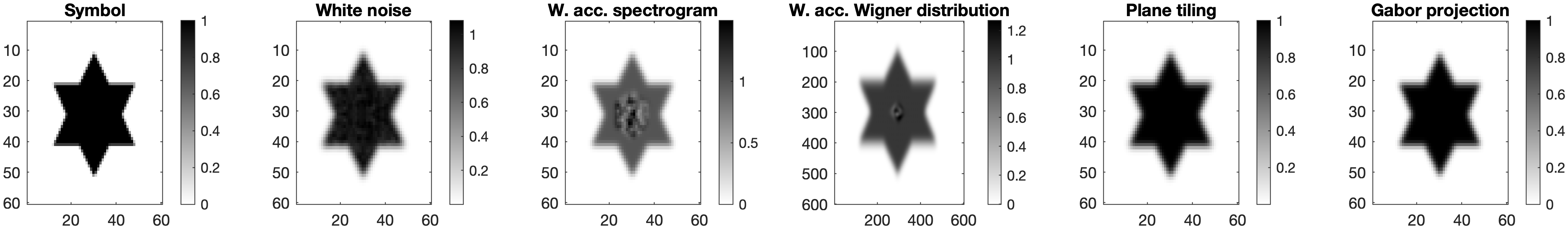}
        \vskip2mm
        \includegraphics[width=6in]{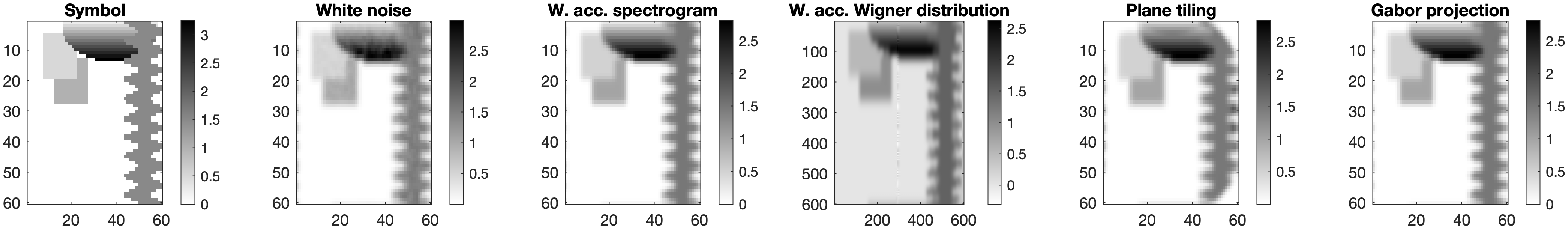}
        \caption{Overview of all methods for a collection of four symbols.}
        \label{fig:overview}
    \end{figure}
    
    \subsection{Implementation details and examples}
    In this section we go through all methods and discuss implementation details.

    \subsubsection{White noise}
    Looking at the formulation of Proposition \ref{prop:white_noise_intermediate} and Theorem \ref{theorem:L1_theorem} in detail, one sees that the process for how the average observed spectrogram approximates $f$ has many intermediate steps. Essentially, $\rho_K \to \vartheta \approx f^2 * |V_gg|^2 \approx f^2$. To highlight the differences between these steps and to show how the estimator handles non-binary symbols, we present plots of all the quantities next to each other below.
    \begin{figure}[H]
        \centering
        \includegraphics[width=6in]{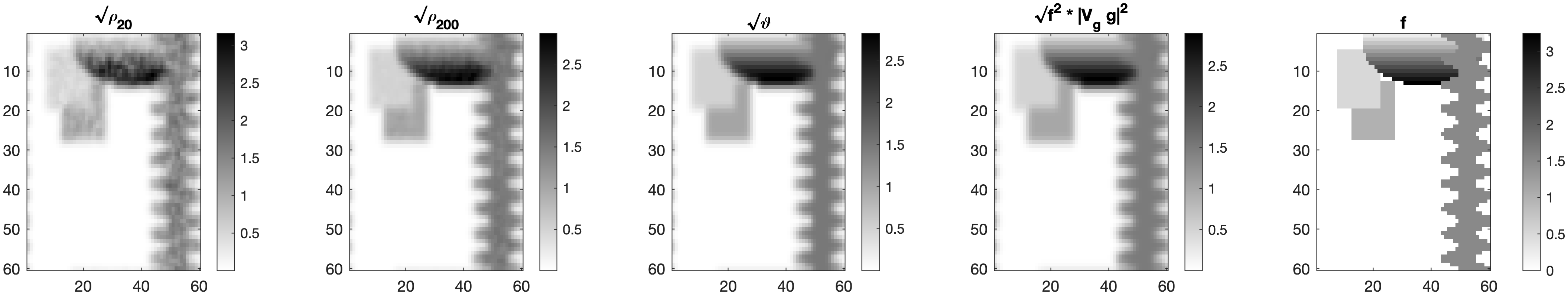}
        \caption{All the intermediate steps in the white noise approximation detailed in Section \ref{sec:white_noise_recovery} for $\varphi = g$. The subscript on $\rho$ indicates the number of samples $K$ used.}
    \end{figure}
    Note in particular that the difference between $\vartheta$ and $f^2 * |V_g^g|$ which corresponds to the error estimate in Lemma \ref{lemma:E_1S_1_bound} is negligible in the above example. However, the averaging of $200$ observations in $\rho_{200}$ leads to a visible difference from $\vartheta$ which is not present in the limit of an infinite number of observations. The conclusion here is that the convergence, which is shown to have $L^1$ error proportional to $\frac{1}{\sqrt{K}}$ in Theorem \ref{theorem:L1_theorem} where $K$ is the number of observations, really is slow to converge in practice.

    Since our implicit estimator for $|f|$ in Proposition \ref{prop:white_noise_intermediate} and Theorem \ref{theorem:L1_theorem} is $\sqrt{\frac{\rho}{\sigma^2}}$, we need some knowledge of $\sigma^2$ in order to be able to scale our estimator correctly. If we have the ability to inspect the noise realizations $(\mathcal{N}_k)_{k=1}^K$, this is straight-forward, either via the variance of each $\mathcal{N}_k$ or as the average value of $|V_\varphi(\mathcal{N}_k)|^2$ as can be seen from the calculations in the proof of Lemma \ref{lemma:prob_geq_t}. Note however that the variance $\sigma^2$ affects our estimator linearly so even without an estimate for $\sigma^2$ we can correctly estimate $|f|$ up to a constant factor.

    In \cite{Romero2022}, there is no need to estimate the variance explicitly as the mask is assumed to be binary. Implicitly, they use the maximum value of the estimator $\Vert \rho \Vert_\infty$ as an estimate of the variance as they use this to normalize $\rho$.
    
    \subsubsection{Weighted accumulated spectrogram}\label{sec:num_acc_spec}
    In the following example, we set both the reconstruction and original window to be the standard Gaussian for convenience. Consequently, the blurring kernel or equivalently the impulse response of the system $f \mapsto \sum_m \lambda_m |V_g h_m|^2$ is given by a two-dimensional Gaussian as this is the STFT of a Gaussian with respect to itself. In the finite setting, it is easy to also find the impulse response of the system by letting the input symbol be a Dirac delta. This approach is employed in the figure below where we have deconvolved the estimator to recover the original symbol precisely as
    $$
        \mathcal{F}^{-1}\left( \frac{\mathcal{F}\left(\sum_m \lambda_m |V_g h_m|^2\right)}{\mathcal{F}\left( \sum_m \lambda_m^0 |V_g h_m^0|^2 \right)} \right)\qquad \text{ where } A_{\delta_0}^g = \sum_m \lambda_m^0 (h_m^0 \otimes h_m^0).
    $$
    \begin{figure}[H]
        \centering
        \includegraphics[width=6in]{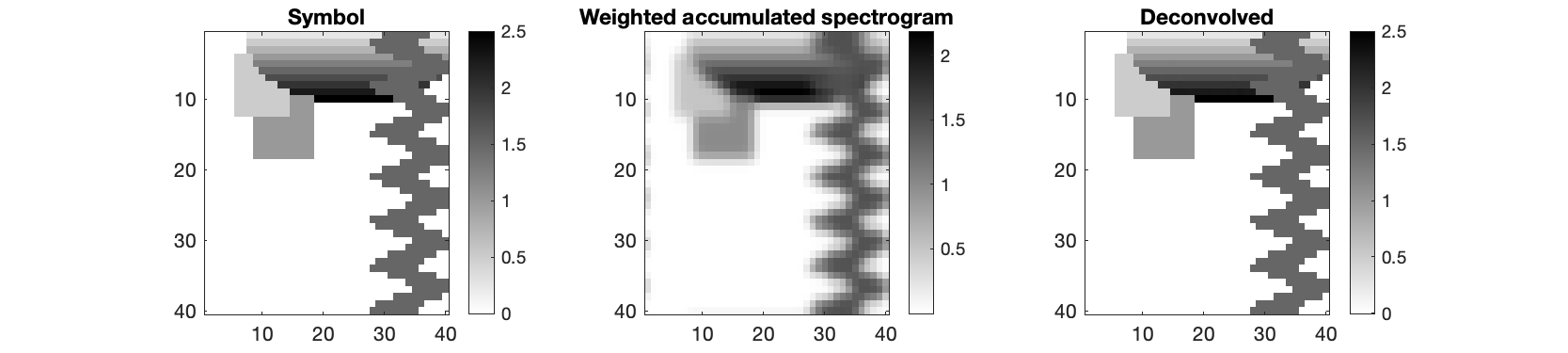}
        \caption{A symbol, the associated accumulated spectrogram and the deconvolved estimate of the symbol.}
    \end{figure}
    This deconvolution procedure is clearly very unstable and if we increase the resolution of the frame we eventually encounter noticeable errors as can be seen below.
    \begin{figure}[H]
        \centering
        \includegraphics[width=6in]{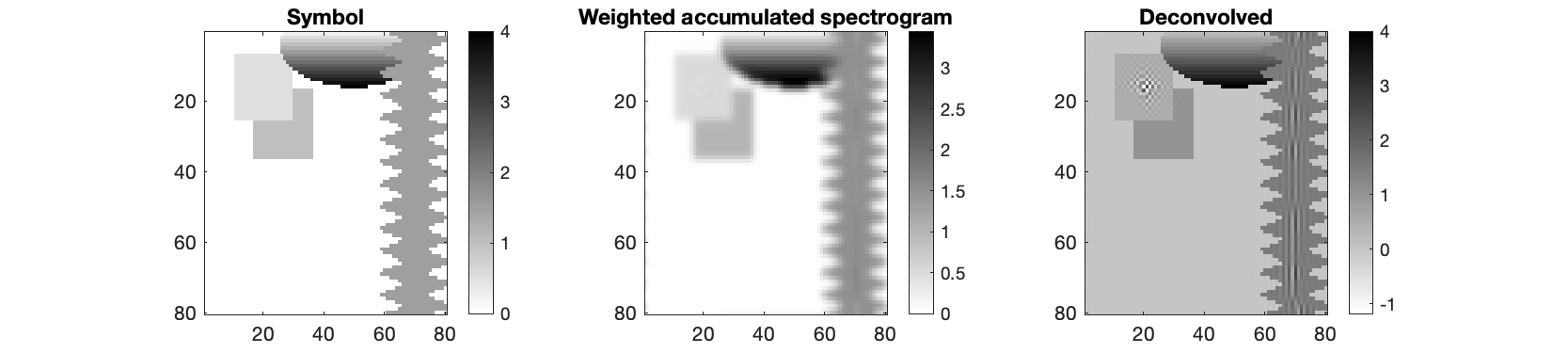}
        \caption{Example of a deconvolution with visible errors.}
    \end{figure}
    The performance of the estimator is also dependent on the orthogonality of the eigenfunctions. In the following example, we have a large amount of eigenfunctions with very similar eigenvalues and the eigenfunctions from LTFAT are not perfectly orthogonal, resulting the visible errors.
    \begin{figure}[H]
        \centering
        \includegraphics[width=6in]{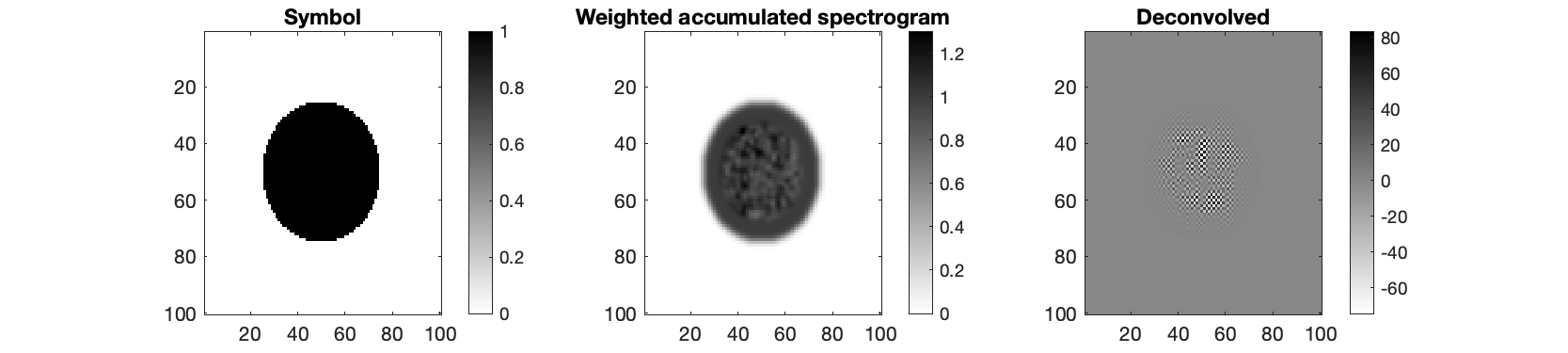}
        \caption{Example of a failed deconvolution due to non-orthogonal eigenfunctions.}
        \label{fig:acc_spec_circle_error}
    \end{figure}
    
    \subsubsection{Weighted accumulated Wigner distribution}\label{sec:num_acc_wig}
    The Wigner-based approach in Theorem \ref{theorem:fourier_deconvolution} is notably more direct than the spectrogram approach in Theorem \ref{theorem:accumulated} in that there is no reconstruction window. Numerically however, there is some ambiguity as to how the integral defining $W(\phi)$ \eqref{eq:wigner_def} should be discretized. The implementations in LTFAT and MATLAB \cite{MathWorksWVD} differ but provided the symbol $f$ is only supported on the positive frequency half of phase space, we get similar results. In the implementation, we employ a procedure which compresses the matrix defining the symbol so that it is zeroed out for negative frequency components to avoid these issues.
    \begin{figure}[H]
        \centering
        \includegraphics[width=6in]{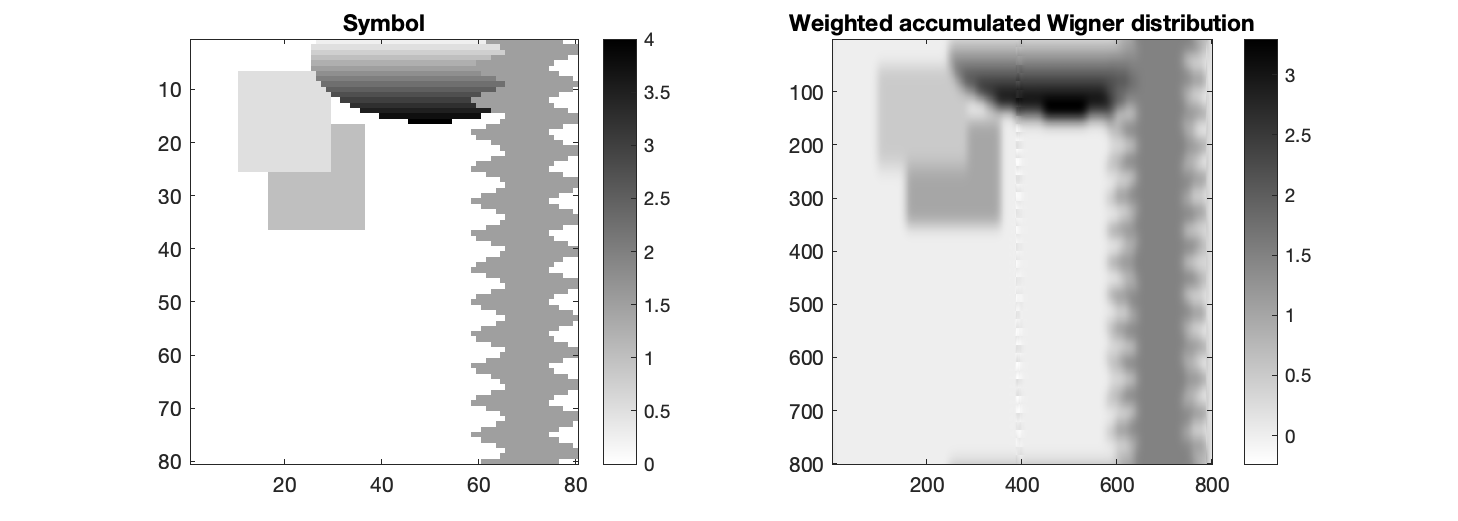}
        \caption{A symbol and the corresponding weighted accumulated Wigner distribution estimator.}
    \end{figure}
    Due to the aforementioned discretization problem and the positive frequency restriction, finding the system impulse response to perform a deconvolution is not as straight-forward as in the weighted accumulated spectrogram case and we do not attempt it.

    Note that in the example above and those in Figure \ref{fig:overview}, the weighted accumulated Wigner distribution is considerably more blurry in the frequency direction than in the time direction. For both cases, the window function is a standard Gaussian.
    
    \subsubsection{Plane tiling}\label{sec:plane_tiling_numerics}
    As discussed briefly near the end of Section \ref{sec:plane_tiling}, by time-frequency shifting a collection of Hermite functions we can get an orthonormal basis which has an appropriate center in the time-frequency plane. This is illustrated in the following figure where we use different number of basis elements to highlight the incremental nature of the approximation.
    \begin{figure}[H]
        \centering
        \includegraphics[width=5.5in]{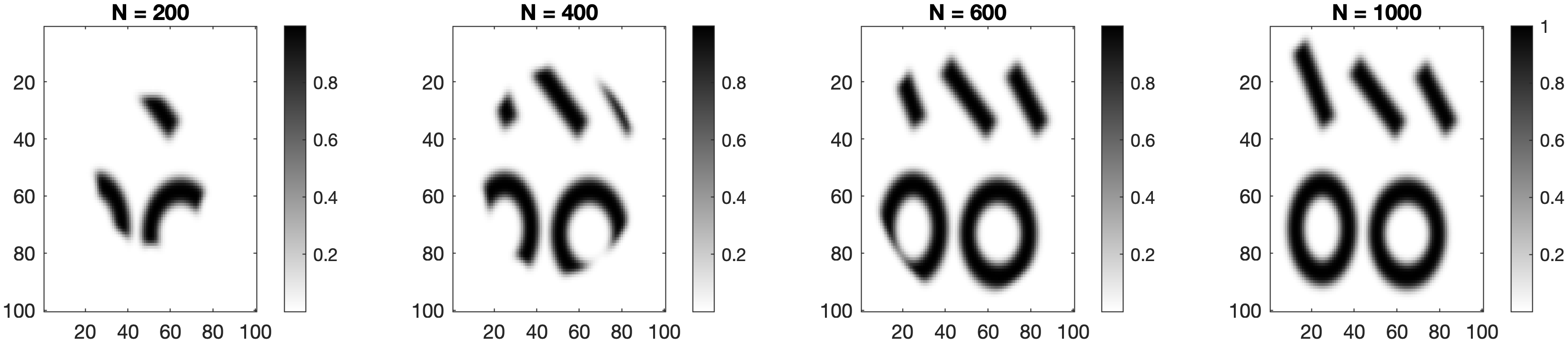}
        \caption{The plane tiling estimator for the same orthonormal basis with different number of terms $N$.}
    \end{figure}
    As each spectrogram has total $L^1$ energy $1$, the number of required basis elements is dependent on the size of the support of the symbol and the resolution of the frame.

    To obtain an orthonormal basis which tiles the plane differently from the Hermites, one can use the eigenfunctions of a localization operator with a prescribed symbol.
    
    \subsubsection{Gabor projection}\label{sec:gabor_projection_details}
    As discussed in the introduction and as is clear from the formulation of Theorem \ref{theorem:gabor_projection}, Gabor projection yields an identical estimator to the weighted accumulated spectrogram. From the numerical side, the main differences are that we can choose which regions of the symbol we want to recover. Moreover, since we are not relying on the eigendata of the operator, we are not susceptible to the type of problems that lead to the issue highlighted in Figure \ref{fig:acc_spec_circle_error} above. This allows us to perform deconvolution for a wider set of symbols and lattice parameters as can be seen by comparing Figure \ref{fig:acc_spec_circle_error} with the one below. 
    \begin{figure}[H]
        \centering
        \includegraphics[width=7in]{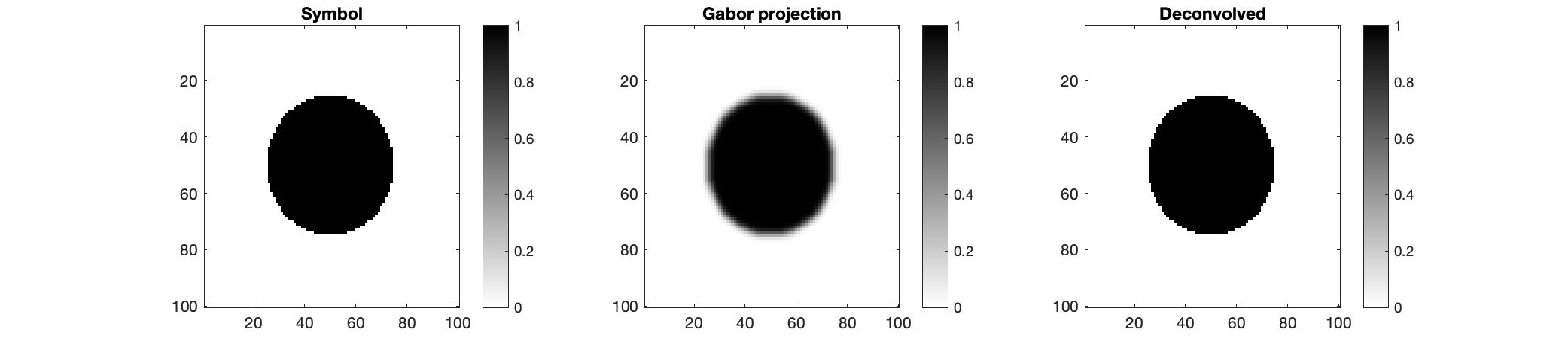}
        \caption{Example of Gabor projection at high frame resolution with a circle as a symbol and a deconvolved version.}
    \end{figure}

    \subsection{Performance comparison}
    Lastly we numerically investigate the $L^1$ error for the different methods. For an estimator $\rho$, we compute the error as 
    $$
    E = \frac{\Vert \rho - f \Vert_{L^1}}{\Vert f \Vert_{L^1}}
    $$
    where the $\Vert f \Vert_{L^1}$ denominator is for normalization. As was discussed in \cite{Abreu2015} and as is the case for all estimators discussed in this paper, the estimators mostly differ from the symbol by some sort of blurring kernel which is constant in size in the sense that if $f$ is dilated, the blurring kernel has a smaller effect. In the discrete setting, dilating the symbol is equivalent to increasing the resolution of the time-frequency frame and so for small lattice parameters we should expect smaller errors. For this reason, all of the errors reported below are for the same time-frequency frame and the absolute sizes of the errors should not be given too much weight. Instead, one should look at how the different methods compare. 
    \begin{table}[H]
        \caption{Reconstruction $L^1$ error for different methods and symbols. Errors are reported in percent of $L^1$ energy of the symbol. Abbreviations: \textbf{WN} = White Noise, \textbf{WAS} = Weighted Accumulated Spectrogram, \textbf{WAWD} = Weighted Accumulated Wigner Distribution, \textbf{PT} = Plane Tiling, \textbf{GP} = Gabor Projection}
        \begin{tabular}{l|lllll}
            \toprule
            \textbf{Symbol}   & \textbf{WN}  & \textbf{WAS}  & \textbf{WAWD}  & \textbf{PT} & \textbf{GP}\\ \midrule
            Circle & $14.8$ & $13.4$ & $15.7$ & $12.1$ & $10.0$\\
            Sum of Gaussians & $7.2$ & $4.5$ & $20.1$ & $6.2$ & $4.5$\\
            Star & $13.7$ & $13.1$ & $16.0$ & $12.0$ & $9.7$\\
            Lines \& circles  & $30.0$ & $23.5$ & $35.8$ & $29.6$ & $23.5$\\
            Blurred lines \& circles & $22.9$ & $17.4$ & $32.8$ & $22.8$ & $17.4$\\
            Tiles & $29.3$ & $23.7$ & $29.1$ & $28.1$ & $23.6$\\
            NTNU & $16.9$ & $13.1$ & $19.2$ & $15.3$ & $13.1$\\\bottomrule
        \end{tabular}
        \label{table:performance}
    \end{table}
    By comparing the errors for Gabor projection with the accumulated spectrogram we see the benefits of Gabor projection discussed in Section \ref{sec:gabor_projection_details} above. The reason the errors for plane tiling are generally lower than those for white noise is that we only use $K = 200$ samples.
    
    All of the symbols used in the comparison are collected in the figure below.
    \begin{figure}[H]
        \centering
        \includegraphics[width=5.8in]{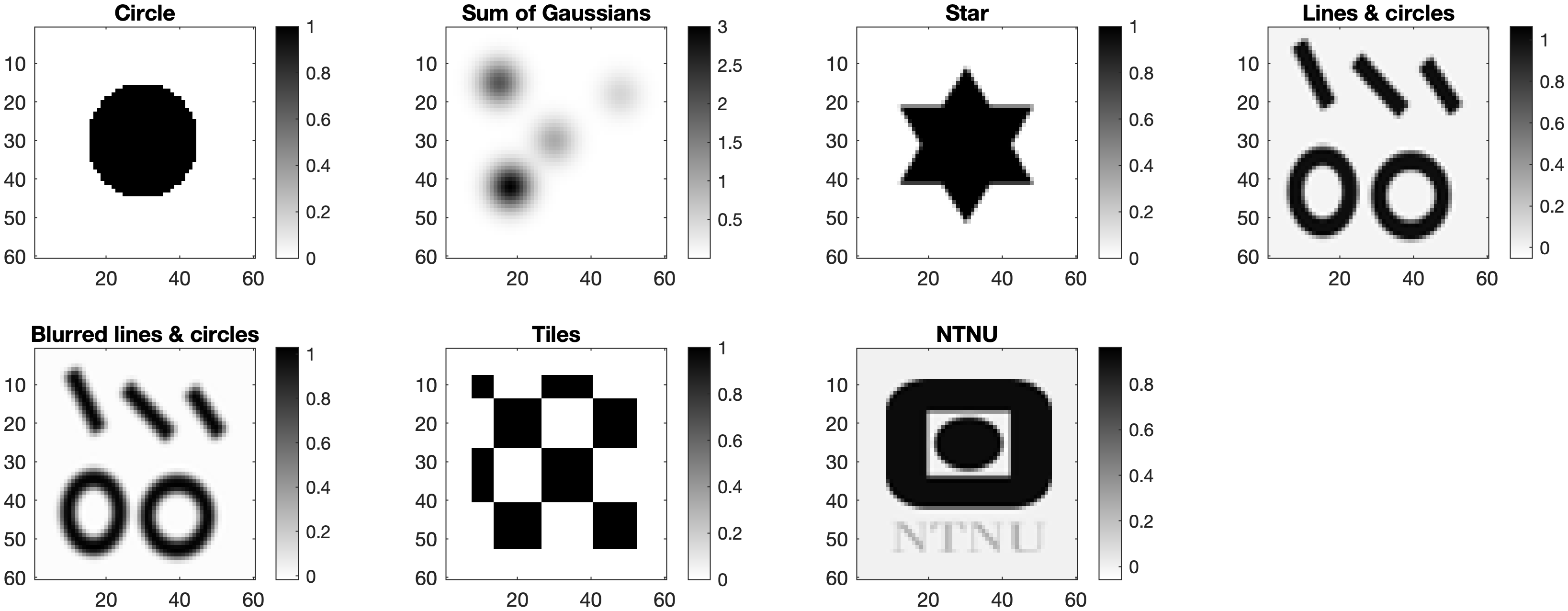}
        \caption{The symbols used for the comparison in Table \ref{table:performance}.}
    \end{figure}

    \subsection*{Acknowledgements}
    This project was partially supported by the Project Pure Mathematics in Norway, funded by Trond Mohn Foundation and Tromsø Research Foundation.

    \printbibliography

\end{document}